\documentclass[platex]{amsart}
\usepackage{amssymb,amsmath,amsthm,empheq,comment,color}
\usepackage[utf8]{inputenc}
\usepackage{enumerate}

\usepackage[top=30truemm,bottom=30truemm,left=20truemm,right=20truemm]{geometry}

\theoremstyle{definition}
\newtheorem{theorem}{Theorem}[section]
\newtheorem{proposition}[theorem]{Proposition}

\newtheorem{lemma}[theorem]{Lemma}
\newtheorem{corollary}[theorem]{Corollary}
\newtheorem{remark}[theorem]{Remark}
\newtheorem*{remark*}{Remark}

\DeclareMathOperator{\re}{Re}
\DeclareMathOperator{\im}{Im}

\DeclareMathOperator{\Mod}{Mod}
\DeclareMathOperator{\ModOp}{\Mod_{\mathrm{op}}}

\title[Minimal mass blow-up solutions for NLS with a potential]{Minimal mass blow-up solutions for nonlinear Schr\"{o}dinger equations with a potential}
\author[N. Matsui]{Naoki Matsui}
\date{\today}
\address[N. Mastui]{Department of Mathematics\\ Tokyo University of Science\\ 1-3 Kagurazaka, Shinjuku-ku, Tokyo 162-8601, Japan}
\email[N. Matsui]{1120703@ed.tus.ac.jp}


\begin{document}
\maketitle

\begin{abstract}
We consider a mass critical nonlinear Schr\"{o}dinger equation with a real-valued potential. In this work, we construct a minimal mass solution that blows up at finite time, under weaker assumptions on spatial dimensions and potentials than Banica, Carles, and Duyckaerts (2011). Moreover, we show that the blow-up solution converges to a blow-up profile. Furthermore, we improve some parts of the arguments in Rapha\"{e}l and Szeftel (2011) and Le Coz, Martel, and Rapha\"{e}l (2016).
\end{abstract}

\section{Introduction}
We consider the following nonlinear Schr\"{o}dinger equation:
\begin{align}
\label{NLS}
i\frac{\partial u}{\partial t}+\Delta u+g(x)|u|^{\frac{4}{N}}u-V(x)u=0,\quad (t,x)\in\mathbb{R}\times\mathbb{R}^N,
\end{align}
where $N\in\mathbb{N}$, $g$ is a real-valued function, and $V$ is a real-valued potential. It is well known that if
\begin{align}
\label{pote1}
V&\in L^{p}(\mathbb{R}^N)+L^{\infty}(\mathbb{R}^N)\quad \left(p\geq 1\ \mathrm{and}\ p>\frac{N}{2}\right),\\
\label{g1}
g&\in L^\infty(\mathbb{R}^N),
\end{align}
then \eqref{NLS} is locally well-posed in $H^1(\mathbb{R}^N)$ (see, e.g., \cite{CSSE}). This means that for any initial value $u_0\in H^1(\mathbb{R}^N)$, there exists a unique maximal solution $u\in C((T_*,T^*),H^1(\mathbb{R}^N))\cap C^1((T_*,T^*),H^{-1}(\mathbb{R}^N))$ for \eqref{NLS} with $u(0)=u_0$. Moreover, the mass (i.e., $L^2$-norm) and energy $E$ of the solution $u$  are conserved by the flow, where 
\[
E(u):=\frac{1}{2}\left\|\nabla u\right\|_2^2-\frac{1}{2+\frac{4}{N}}\int_{\mathbb{R}^N}g(x)|u(x)|^{2+\frac{4}{N}}dx+\frac{1}{2}\int_{\mathbb{R}^N}V(x)|u(x)|^2dx.
\]
Furthermore, the blow-up alternative holds:
\[
T^*<\infty\quad \mbox{implies}\quad \lim_{t\nearrow T^*}\left\|\nabla u(t)\right\|_2=\infty.
\]

We define $\Sigma^k$ by
\[
\Sigma^k:=\left\{u\in H^k\left(\mathbb{R}^N\right)\ \middle|\ |x|^ku\in L^2\left(\mathbb{R}^N\right)\right\},\quad \|u\|_{\Sigma^k}^2:=\|u\|_{H^k}^2+\||x|^ku\|_2^2.
\]
Particularly, $\Sigma^1$ is called the virial space. If $u_0\in \Sigma^1$, then the solution $u$ for \eqref{NLS} with $u(0)=u_0$ belongs to $C((T_*,T^*),\Sigma^1(\mathbb{R}^N))$.

Moreover, we consider
\begin{align}
\label{pote1'}
V\in L^{p}(\mathbb{R}^N)+L^{\infty}(\mathbb{R}^N)\quad \left(p\geq 2\ \mathrm{and}\ p>\frac{N}{2}\right).
\end{align}
If $u_0\in \Sigma^2$, then the solution $u$ for \eqref{NLS} with $u(0)=u_0$ belongs to $C((T_*,T^*),\Sigma^2(\mathbb{R}^N))\cap C^1((T_*,T^*),L^2(\mathbb{R}^N))$ and $|x|\nabla u\in C((T_*,T^*),L^2(\mathbb{R}^N))$.

In this paper, we investigate conditions for the potential related with the existence of a minimal mass blow-up solution for \eqref{NLS}.

\subsection{The case $V=0$ and $g=1$}
\label{sec:Vconst}
Firstly, we describe the results when $V$ is a real constant and $g=1$. Let $u_V$ be a solution for \eqref{NLS} and define $u(t,x):=u_V(t,x)e^{iVt}$. Then $u$ is a solution for \eqref{NLS} with $V=0$ and $g=1$. Therefore, we may assume that $V=0$, that is, we consider
\begin{align}
\label{NLS0}
i\frac{\partial u}{\partial t}+\Delta u+|u|^{\frac{4}{N}}u=0,\quad (t,x)\in\mathbb{R}\times\mathbb{R}^N.
\end{align}

It is well known (\cite{BLGS,KGS,WGS}) that there exists a unique classical solution $Q$ for
\[
-\Delta Q+Q-\left|Q\right|^{\frac{4}{N}}Q=0,\quad Q\in H^1(\mathbb{R}^N),\quad Q>0,\quad Q\mathrm{\ is\ radial},
\]
which is called the ground state. If $\|u\|_2=\|Q\|_2$ ($\|u\|_2<\|Q\|_2$, $\|u\|_2>\|Q\|_2$), we say that $u$ has the \textit{critical mass} (\textit{subcritical mass}, \textit{supercritical mass}, respectively).

We note that $E_{\mathrm{crit}}(Q)=0$, where $E_{\mathrm{crit}}$ is the energy when $V=0$. Moreover, the ground state $Q$ attains the best constant in the Gagliardo-Nirenberg inequality
\[
\left\|v\right\|_{2+\frac{4}{N}}^{2+\frac{4}{N}}\leq\left(1+\frac{2}{N}\right)\left(\frac{\left\|v\right\|_2}{\left\|Q\right\|_2}\right)^{\frac{4}{N}}\left\|\nabla v\right\|_2^2\quad\mbox{for }v\in H^1(\mathbb{R}^N).
\]
Therefore, for all $v\in H^1(\mathbb{R}^N)$,
\[
E_{\mathrm{crit}}(v)\geq \frac{1}{2}\left\|\nabla v\right\|_2^2\left(1-\left(\frac{\left\|v\right\|_2}{\left\|Q\right\|_2}\right)^{\frac{4}{N}}\right)
\]
holds. This inequality and the mass and energy conservations imply that any subcritical mass solution for \eqref{NLS0} exists globally in time and is bounded in $H^1(\mathbb{R}^N)$.

Regarding the critical mass case, we apply the pseudo-conformal transformation
\[
u(t,x)\ \mapsto\ \frac{1}{\left|t\right|^\frac{N}{2}}u\left(-\frac{1}{t},\pm\frac{x}{t}\right)e^{i\frac{\left|x\right|^2}{4t}}
\]
to the solitary wave solution $u(t,x):=Q(x)e^{it}$. Then we obtain
\[
S(t,x):=\frac{1}{\left|t\right|^\frac{N}{2}}Q\left(\frac{x}{t}\right)e^{-\frac{i}{t}}e^{i\frac{\left|x\right|^2}{4t}},
\]
which is also a solution for \eqref{NLS0} and satisfies
\[
\left\|S(t)\right\|_2=\left\|Q\right\|_2,\quad \left\|\nabla S(t)\right\|_2\sim\frac{1}{\left|t\right|}\quad (t\nearrow 0).
\]
Namely, $S$ is a minimal mass blow-up solution for \eqref{NLS0}. Moreover, $S$ is the only finite time blow-up solution for \eqref{NLS0} with critical mass, up to the symmetries of the flow (see \cite{MMMB}).

Regarding the supercritical mass case, there exists a solution $u$ for \eqref{NLS0} such that
\[
\left\|\nabla u(t)\right\|_2\sim\sqrt{\frac{\log\bigl|\log\left|T^*-t\right|\bigr|}{T^*-t}}\quad (t\nearrow T^*)
\]
(see \cite{MRUPB,MRUDB}).

\subsection{Previous results}
Banica, Carles, and Duyckaerts \cite{BCD} present the following result for \eqref{NLS}.

\begin{theorem}[\cite{BCD}]
\label{BCD}
Let $N=1$ or $2$, $V\in C^2(\mathbb{R}^N,\mathbb{R})$, and $g\in C^4(\mathbb{R}^N,\mathbb{R})$. Assume $\left(\frac{\partial}{\partial x}\right)^\beta V\in L^\infty(\mathbb{R}^N)\ (|\beta|\leq 2)$, $\left(\frac{\partial}{\partial x}\right)^\beta g\in L^\infty(\mathbb{R}^N)\ (|\beta|\leq 4)$, and
\[
g(0)=1,\quad \frac{\partial g}{\partial x_j}(0)=\frac{\partial^2 g}{\partial x_j\partial x_k}(0)=0\quad (1\leq j,k\leq N).
\]
Then there exist $T>0$ and a solution $u\in C((0,T),\Sigma^1)$ for \eqref{NLS} such that
\[
\left\|u(t)-\frac{1}{\lambda(t)^\frac{N}{2}}Q\left(\frac{x-x(t)}{\lambda(t)}\right)e^{i\frac{|x|^2}{4t}-i\theta\left(\frac{1}{t}\right)-itV(0)}\right\|_{\Sigma^1}\rightarrow 0\quad (t\searrow 0),
\]
where $\theta$ and $\lambda$ are continuous real-valued functions and $x$ is a continuous $\mathbb{R}^N$-valued function such that
\begin{align*}
&\theta(\tau)=\tau+o(\tau)\quad\mbox{as }\tau\rightarrow+\infty,\\
&\lambda(t)\sim t\mbox{ and }|x(t)|=o(t)\quad\mbox{as }t\searrow 0.
\end{align*}
\end{theorem}

This result means that if $V$ and $g$ are sufficiently smooth and bounded, then there exists a minimal mass solution that blows up at finite time with a blow-up rate $|t|^{-1}$. The blow-up rate is identical with the blow-up rate when $g=1$ and $V=0$.

Le Coz, Martel, and Rapha\"{e}l \cite{LMR} present the following result for
\begin{align}
\label{DPNLS}
i\frac{\partial u}{\partial t}+\Delta u+|u|^{\frac{4}{N}}u+|u|^{p-1}u=0,\quad (t,x)\in\mathbb{R}\times\mathbb{R}^N.
\end{align}

\begin{theorem}[\cite{LMR}]
Let $N=1,2,3$ and $1<p<1+\frac{4}{N}$. Then for any energy level $E_0\in\mathbb{R}$, there exist $t_0<0$ and a radially symmetric initial value $u_0\in H^1(\mathbb{R}^N)$ with
\[
\|u_0\|_2=\|Q\|_2,\quad E(u_0)=E_0
\]
such that the corresponding solution $u$ for \eqref{DPNLS} with $u(t_0)=u_0$ blows up at $t=0$ with a blow-up rate of
\[
\|\nabla u(t)\|_2=\frac{C(p)+o_{t\nearrow 0}(t)}{|t|^{\sigma}},
\]
where $\sigma=\frac{4}{4+N(p-1)}$ and $C(p)>0$.
\end{theorem}

This result means that the attractive nonlinearity $|u|^{p-1}u$ affects blow-up rates of blow-up solutions with critical mass. Moreover, for any energy level, there exists a blow-up solution with critical mass and the energy.

\subsection{Main result}
For the potential $V$, we consider the following:
\begin{align}
\label{pote2}
&V\in C^{1,1}_{\mathrm{loc}}\mathbb{R}^N,\\
\label{V1growth}
&\nabla V=O(|x|),\\
\label{V2growth}
&\nabla^2 V=O(|x|^r)\quad \mbox{for some }r\geq 0.
\end{align}
For the function $g$, we consider the following:
\begin{align}
\label{g2}
&g\in C^{3,1}_{\mathrm{loc}}\mathbb{R}^N,\\
\label{gflat}
&g(0)=1,\quad \frac{\partial g}{\partial x_j}(0)=\frac{\partial^2 g}{\partial x_j\partial x_k}(0)=0\quad (1\leq j,k\leq N),\\
\label{gint}
&g,\nabla g,x\cdot\nabla g\in L^\infty(\mathbb{R}^N),\\
\label{ggrowth}
&\nabla^3g,\nabla^4g=O(|x|^{r_g})\quad \mbox{for some }r_g\geq 0
\end{align}
The main result of this paper is the following, which gives an extension of Theorem \ref{BCD}.

\begin{theorem}[Existence of a minimal mass blow-up solution]
\label{theorem:EMBS}
Let the potential $V$ satisfy \eqref{pote1'}, \eqref{pote2}, \eqref{V1growth}, and \eqref{V2growth}. Let the function $g$ satisfy \eqref{g2}, \eqref{gflat}, \eqref{gint}, and \eqref{ggrowth}. Then there exist $t_0<0$ and a radial initial value $u_0\in \Sigma^1$ with $\|u_0\|_2=\|Q\|_2$ such that the corresponding solution $u$ for \eqref{NLS} with $u(t_0)=u_0$ blows up at $t=0$. Moreover,
\[
\left\|u(t,x)-\frac{1}{\lambda(t)^\frac{N}{2}}Q\left(\frac{x+w(t)}{\lambda(t)}\right)e^{-i\frac{b(t)}{4}\frac{|x+w(t)|^2}{\lambda(t)^2}+i\gamma(t)}\right\|_{\Sigma^1}\rightarrow 0\quad (t\nearrow 0)
\]
holds for some $C^1$ functions $\lambda:(t_0,0)\rightarrow(0,\infty)$, $b,\gamma:(t_0,0)\rightarrow\mathbb{R}$, and $w:(t_0,0)\rightarrow\mathbb{R}^N$ such that
\[
\lambda(t)=|t|\left(1+o(1)\right),\quad b(t)=|t|\left(1+o(1)\right),\quad \gamma(t)\sim |t|^{-1},\quad |w(t)|=o(|t|)
\]
as $t\nearrow 0$.
\end{theorem}

\begin{remark}
In contrast, if $V$ satisfies \eqref{pote1}, then any subcritical mass solution for \eqref{NLS} exists globally in time and is bounded in $H^1(\mathbb{R}^N)$. This can be proved easily by the Gagliardo-Nirenberg inequality and the Sobolev embedding theorem. Therefore, the solution in Theorem \ref{theorem:EMBS} is a minimal mass blow-up solution.
\end{remark}

\begin{remark}
Since begin locally Lipschitz continuous and belonging to $W^{1,\infty}_{\mathrm{loc}}$ are equivalent, if $V$ satisfies \eqref{pote2}, then $\nabla V$ and $\nabla^2 V$ are bounded near the origin.
\end{remark}

\subsection{Outline of proof}
We prove Theorem \ref{theorem:EMBS} by using a simplified version with modification of the method of Le Coz, Martel, and Rapha\"{e}l \cite{LMR}, which is based on seminal work of Rapha\"{e}l and Szeftel \cite{RSEU}. We proceed in the following steps:

\begin{itemize}
\item[Step 1.] For a solution $u$ for \eqref{NLS}, we consider the following transformation:
\[
u(t,x)=\frac{1}{\lambda(s)^\frac{N}{2}}v\left(s,y\right)e^{-i\frac{b(s)|y|^2}{4}+i\gamma(s)},\quad y=\frac{x+w(s)}{\lambda(s)},\quad \frac{ds}{dt}=\frac{1}{\lambda(s)^2}.
\]
\item[Step 2.] Let $v=Q+\varepsilon$ for some error function $\varepsilon$. Then we obtain the equation of $\varepsilon$ (Lemmas \ref{decomposition} and \ref{theorem:epsieq}):
\begin{align*}
0=&\ i\frac{\partial \varepsilon}{\partial s}+\Delta \varepsilon-\varepsilon+g(\lambda y-w)|Q+\varepsilon|^\frac{4}{N}(Q+\varepsilon)-Q^{1+\frac{4}{N}}-\lambda^2 V(\lambda y-w)\varepsilon\\
&\hspace{20pt}+\mbox{modulation terms}+\mbox{an error term}.
\end{align*}
\item[Step 3.] By using the modulation terms and $\varepsilon$, we estimate the parameters $\lambda$, $b$, $\gamma$, and $w$ (from Section \ref{sec:uniesti} to Section \ref{sec:convesti}).
\item[Step 4.] We construct a sequence of suitable solutions for \eqref{NLS} and show that the limit of the sequence is the desired minimal mass blow-up solution (Section \ref{sec:proof}).
\end{itemize}

\subsection{Comments on Theorem \ref{theorem:EMBS}}
Firstly, the assumptions in Theorem \ref{theorem:EMBS} are weaker than those in Theorem \ref{BCD}. Theorem \ref{theorem:EMBS} has no restrictions on spatial dimensions. On the other hand, according to the lack of regularity of the nonlinearity $|u|^{\frac{4}{N}}u$, Theorem \ref{BCD} requires the restriction $N=1$ or $2$. Although Theorem \ref{theorem:EMBS} is also affected by the lack of regularity, we overcome this difficulty by using the properties of the ground state. Regarding potentials, Theorem \ref{theorem:EMBS} requires less differentiability and integrability than Theorem \ref{BCD}. Indeed, any $V\in C^2(\mathbb{R}^N)\cap W^{2,\infty}(\mathbb{R}^N)$ satisfies \eqref{pote1'}, \eqref{pote2}, \eqref{V1growth}, and \eqref{V2growth}. On the other hand, there exists $V$ such that it satisfies \eqref{pote1'}, \eqref{pote2}, \eqref{V1growth}, and \eqref{V2growth} but does not satisfy the assumption in Theorem \ref{BCD}, e.g., 
\[
V(x):=\frac{1}{1+x^2}\cos(x^4).
\]
When $V$ does not satisfy \eqref{pote2}, blow-up rates should change as the result of Le Coz, Martel, and Rapha\"{e}l \cite{LMR}. Regarding the function $g$, Theorem \ref{theorem:EMBS} does not require higher-order derivatives to be bounded. For example, Theorem \ref{theorem:EMBS} applies to the following function:
\[
g(x):=\frac{1}{1+x^4}\cos(x^4).
\]

Secondly, we improve some parts of the arguments in Le Coz, Martel, and Rapha\"{e}l \cite{LMR} and Rapha\"{e}l and Szeftel \cite{RSEU}. Although the authors of \cite{LMR,RSEU} introduce the Morawetz functional (\cite[Section 5]{LMR} and \cite[Lemma 3.3]{RSEU}) and apply a \textit{truncation} procedure to the functional, we avoid using the functional by modifying the definition of $\varepsilon$. As a result, without the truncation, we work directly in the virial space $\Sigma^1$. Moreover, the authors of \cite{LMR} use the continuous dependence on the initial value for \eqref{DPNLS} in $H^s(\mathbb{R}^N)$ for some $s\in [0,1)$. Although this continuous dependence is an important fact in the proof of the main result in \cite{LMR}, it is not obvious for \eqref{NLS}. Therefore, instead of proving the continuous dependence for \eqref{NLS} in $H^s(\mathbb{R}^N)$ for some $s\in [0,1)$, we use Lemma \ref{contidepend} in Appendix \ref{sec:SFacts}, which gives a kind of the continuous dependence. Consequently, we provide a simpler and more general proof.

\section{Notation and preliminaries}
\label{sec:Preliminaries}
We define
\begin{align*}
&(u,v)_2:=\re\int_{\mathbb{R}^N}u(x)\overline{v}(x)dx,\quad \left\|u\right\|_p:=\left(\int_{\mathbb{R}^N}|u(x)|^pdx\right)^\frac{1}{p},\\
&f(z):=|z|^\frac{4}{N}z,\quad  F(z):=\frac{1}{2+\frac{4}{N}}|z|^{2+\frac{4}{N}}\quad \mbox{for $z\in\mathbb{C}$}.
\end{align*}
By identifying $\mathbb{C}$ with $\mathbb{R}^2$, we denote the differentials of $f$ and $F$ by $df$ and $dF$, respectively. We define
\[
\Lambda:=\frac{N}{2}+x\cdot\nabla,\quad L_+:=-\Delta+1-\left(1+\frac{4}{N}\right)Q^\frac{4}{N},\quad L_-:=-\Delta+1-Q^\frac{4}{N}.
\]
Namely, $\Lambda$ is the generator of $L^2$-scaling, and $L_+$ and $L_-$ come from the linearised Schr\"{o}dinger operator to close $Q$. Then
\[
L_-Q=0,\quad L_+\Lambda Q=-2Q,\quad L_-|x|^2Q=-4\Lambda Q,\quad L_+\rho=|x|^2 Q,\quad L_-xQ=-\nabla Q
\]
hold, where $\rho\in\mathcal{S}(\mathbb{R}^N)$ is the unique radial solution for $L_+\rho=|x|^2 Q$. Furthermore, there exists $\mu>0$ such that for any $u\in H^1(\mathbb{R}^N)$,
\begin{align}
\label{Lcoer}
&\left\langle L_+\re u,\re u\right\rangle+\left\langle L_-\im u,\im u\right\rangle\nonumber\\
\geq&\ \mu\left\|u\right\|_{H^1}^2-\frac{1}{\mu}\left({(\re u,Q)_2}^2+\left|(\re u,xQ)_2\right|^2+{(\re u,|x|^2 Q)_2}^2+{(\im u,\rho)_2}^2\right)
\end{align}
holds (see, e.g., \cite{MRO,MRUPB,RSEU,WL}). Finally, we use the notation $\lesssim$ and $\gtrsim$ when the inequalities hold up to a positive constant. We also use the notation $\approx$ when $\lesssim$ and $\gtrsim$ hold.

For the ground state $Q$, the following property holds:

\begin{proposition}[E.g., \cite{LMR}]
\label{GSP}
For any multi-index $\alpha$, there exist $C_\alpha,\kappa_\alpha>0$ such that
\[
\left|\left(\frac{\partial}{\partial x}\right)^\alpha Q(x)\right|\leq C_\alpha Q(x),\quad \left|\left(\frac{\partial}{\partial x}\right)^\alpha \rho(x)\right|\leq C_\alpha(1+|x|)^{\kappa_\alpha} Q(x).
\]
\end{proposition}

We estimate the error term $\Psi$ that is defined by
\[
\Psi(y):=\lambda^2V(\lambda y-w)Q(y).
\]
Without loss of generality, we may assume that $V(0)=0$ (see Section \ref{sec:Vconst}).

\begin{proposition}[Estimate of $\Psi$]
\label{Psiesti}
There exists a sufficiently small constant $\epsilon'>0$ such that 
\[
\left\|e^{\epsilon'|y|}\Psi\right\|_2+\left\|e^{\epsilon'|y|}\nabla\Psi\right\|_2\lesssim\lambda^2(\lambda+|w|)
\]
for $0<\lambda\ll 1$ and $w\in\mathbb{R}^N$. Moreover, for any radial function $\varphi\in L^2(\mathbb{R}^N)$,
\[
\left|(\Psi,\varphi)_2\right|\lesssim\lambda^2|w|+\lambda^4
\]
\end{proposition}

\begin{proof}
By using Taylor's theorem and $V(0)=0$, we write
\begin{align*}
\lambda^2V(\lambda y-w)&=\lambda^2(\lambda y-w)\cdot\nabla V(0)+\sum_{|\alpha|=2}\int_0^1 \lambda^2(\lambda y-w)^\alpha\frac{\partial^\alpha V}{\partial x^\alpha}(\tau(\lambda y-w))(1-\tau)d\tau,\\
\lambda^3\frac{\partial V}{\partial x_j}(\lambda y-w)&=\lambda^3\frac{\partial V}{\partial x_j}(0)+\int_0^1\lambda^3(\lambda y-w)\cdot\left(\nabla\frac{\partial V}{\partial x_j}\right)(\tau(\lambda y-w))d\tau.
\end{align*}
Therefore, we have
\begin{align*}
\left|\Psi(y)\right|&\leq\lambda^2(\lambda|y|+|w|)Q(y)+\lambda^2(\lambda|y|+|w|)^{2+r}Q(y),\\
\left|\nabla\Psi(y)\right|&\leq \lambda^3Q(y)+\lambda^2(\lambda|y|+|w|)^{1+r}Q+\lambda^2(\lambda|y|+|w|)|\nabla Q(y)|+\lambda^2(\lambda|y|+|w|)^{2+r}|\nabla Q(y)|.
\end{align*}
Therefore, according to Proposition \ref{GSP} and the exponential decay of $Q$ (\cite[Theorem 8.1.1]{CSSE}), there exists a sufficiently small constant $\epsilon'>0$ such that
\[
\left\|e^{\epsilon'|y|}\Psi\right\|_2+\left\|e^{\epsilon'|y|}\nabla\Psi\right\|_2\lesssim\lambda^2(\lambda+|w|).
\]

Since $(yQ,\varphi)_2=0$ for any radial function $\varphi\in L^2(\mathbb{R}^N)$, we obtain
\[
(\Psi,\varphi)_2=-\lambda^2 w\cdot\nabla V(0)(Q,\varphi)_2+\sum_{|\alpha|=2}\int_0^1 \lambda^2\left((\lambda y-w)^\alpha\frac{\partial^\alpha V}{\partial x^\alpha}(\tau(\lambda y-w))Q,\varphi\right)_2(1-\tau)d\tau.
\]
Therefore, we obtain conclusion.
\end{proof}

At the end of this section, we state the following standard result. For the proof, see \cite{MRUPB}.

\begin{lemma}[Decomposition]
\label{decomposition}
There exists $\overline{C}>0$ such that the following statement holds. Let $I$ be an interval and $\delta>0$ be sufficiently small. We assume that $u\in C(I,H^1(\mathbb{R}^N))\cap C^1(I,H^{-1}(\mathbb{R}^N))$ satisfies \[
\forall\ t\in I,\ \left\|\lambda(t)^{\frac{N}{2}}u\left(t,\lambda(t)y-w(t)\right)e^{i\gamma(t)}-Q\right\|_{H^1}< \delta
\]
for some functions $\lambda:I\rightarrow(0,\infty)$, $\gamma:I\rightarrow\mathbb{R}$, and $w:I\rightarrow\mathbb{R}^N$. Then there exist unique functions $\tilde{\lambda}:I\rightarrow(0,\infty)$, $\tilde{b}:I\rightarrow\mathbb{R}$, $\tilde{\gamma}:I\rightarrow\mathbb{R}\slash 2\pi\mathbb{Z}$, and $\tilde{w}:I\rightarrow\mathbb{R}^N$ such that 
\begin{align}
\label{mod}
&u(t,x)=\frac{1}{\tilde{\lambda}(t)^{\frac{N}{2}}}\left(Q+\tilde{\varepsilon}\right)\left(t,\frac{x+\tilde{w}(t)}{\tilde{\lambda}(t)}\right)e^{-i\frac{\tilde{b}(t)}{4}\frac{|x+\tilde{w}(t)|^2}{\tilde{\lambda}(t)^2}+i\tilde{\gamma}(t)},\\
&\left|\frac{\tilde{\lambda}(t)}{\lambda(t)}-1\right|+\left|\tilde{b}(t)\right|+\left|\tilde{\gamma}(t)-\gamma(t)\right|_{\mathbb{R}\slash 2\pi\mathbb{Z}}+\left|\frac{\tilde{w}(t)-w(t)}{\tilde{\lambda}(t)}\right|<\overline{C}\nonumber
\end{align}
hold, where $|\cdot|_{\mathbb{R}\slash 2\pi\mathbb{Z}}$ is defined by
\[
|c|_{\mathbb{R}\slash 2\pi\mathbb{Z}}:=\inf_{m\in\mathbb{Z}}|c+2\pi m|,
\]
and that $\tilde{\varepsilon}$ satisfies the orthogonal conditions
\begin{align}
\label{orthocondi}
\left(\tilde{\varepsilon},i\Lambda Q\right)_2=\left(\tilde{\varepsilon},|y|^2Q\right)_2=\left(\tilde{\varepsilon},i\rho\right)_2=0,\quad \left(\tilde{\varepsilon},yQ\right)_2=0
\end{align}
on $I$. In particular, $\tilde{\lambda}$, $\tilde{b}$, $\tilde{\gamma}$, and $\tilde{w}$ are $C^1$ functions and independent of $\lambda$, $\gamma$, and $w$.
\end{lemma}

\section{Uniformity estimates for modulation terms}
\label{sec:uniesti}
From this section to Section \ref{sec:convesti}, we prepare lemmas for the proof of Theorem \ref{theorem:EMBS}.

Given $t_1<0$ which is sufficiently close to $0$, we define $s_1:=-{t_1}^{-1}$ and $\lambda_1=b_1={s_1}^{-1}$. Let $u(t)$ be the solution for \eqref{NLS} with an initial value
\begin{align}
\label{initial}
u(t_1,x):=\frac{1}{{\lambda_1}^\frac{N}{2}}Q\left(\frac{x}{\lambda_1}\right)e^{-i\frac{b_1}{4}\frac{|x|^2}{{\lambda_1}^2}}.
\end{align}
Note that $u\in C((T_*,T^*),\Sigma^2(\mathbb{R}^N))$ and $|x|\nabla u\in C((T_*,T^*),L^2(\mathbb{R}^N))$. Moreover,
\[
\im\int_{\mathbb{R}^N}u(t_1,x)\nabla\overline{u}(t_1,x)dx=0
\]
holds.

Since $u$ satisfies the assumption in Lemma \ref{decomposition} in a neighbourhood of $t_1$, there exist decomposition parameters $\tilde{\lambda}_{t_1}$, $\tilde{b}_{t_1}$, $\tilde{\gamma}_{t_1}$, $\tilde{w}_{t_1}$, and $\tilde{\varepsilon}_{t_1}$ such that \eqref{mod} and \eqref{orthocondi} hold in the neighbourhood. We define the rescaled time $s_{t_1}$ by
\[
s_{t_1}(t):=s_1-\int_t^{t_1}\frac{1}{\tilde{\lambda}_{t_1}(\tau)^2}d\tau.
\]
Moreover, we define
\begin{align*}
&t_{t_1}:={s_{t_1}}^{-1}, \quad\lambda_{t_1}(s):=\tilde{\lambda}_{t_1}(t_{t_1}(s)),\quad b_{t_1}(s):=\tilde{b}_{t_1}(t_{t_1}(s)),\\
&\gamma_{t_1}(s):=\tilde{\gamma}_{t_1}(t_{t_1}(s)),\quad w_{t_1}(s):=\tilde{w}_{t_1}(t_{t_1}(s)),\quad \varepsilon_{t_1}(s,y):=\tilde{\varepsilon}_{t_1}(t_{t_1}(s),y).
\end{align*}
For the sake of clarity in notation, we often omit the subscript $t_1$. Furthermore, let $I_{t_1}$ be the maximal interval of the existence of the decomposition such that \eqref{mod} and \eqref{orthocondi} hold and we define 
\[
J_{s_1}:=s_{t_1}\left(I_{t_1}\right).
\]
Additionally, let $s_0\ (\leq s_1)$ be sufficiently large and
\[
s':=\max\left\{s_0,\inf J_{s_1}\right\}.
\]

Let $K$ be sufficiently large and $L$ and $M$ satisfy
\[
L=\frac{3}{2}+\frac{1}{K},\quad 1<M<2(L-1).
\]
Moreover, we define $s_*$ by
\[
s_*:=\inf\left\{\sigma\in(s',s_1]\ \middle|\ \mbox{\eqref{bootstrap} holds on }[\sigma,s_1]\right\},
\]
where
\begin{align}
\label{bootstrap}
\left\{\begin{array}{l}
\left\|\varepsilon(s)\right\|_{H^1}^2+b(s)^2\||y|\varepsilon(s)\|_2^2<s^{-2L},\\
\left|s\lambda(s)-1\right|<s^{-M},\quad \left|sb(s)-1\right|<s^{-M},\quad |w(s)|<s^{-\frac{3}{2}}.
\end{array}\right.
\end{align}
Note that for all $s\in(s_*,s_1]$, we have
\[
s^{-1}(1-s^{-M})<\lambda(s),b(s)<s^{-1}(1+s^{-M}).
\]

Finally, we define
\[
\Mod(s):=\left(\frac{1}{\lambda}\frac{\partial \lambda}{\partial s}+b,\frac{\partial b}{\partial s}+b^2,1-\frac{\partial \gamma}{\partial s},\frac{\partial w}{\partial s}\right).
\]
The goal of this section is to estimate of $\Mod(s)$.

In the following, positive constants $C$ and $\epsilon$ are sufficiently large and small, respectively. If necessary, we retake $s_0$ and $s_1$ sufficiently large in  response to $\epsilon$.

\begin{lemma}[The equation for $\varepsilon$]
\label{theorem:epsieq}
On $J_{s_1}$, 
\begin{align}
\label{epsieq}
\Psi&=i\frac{\partial \varepsilon}{\partial s}+\Delta \varepsilon-\varepsilon+g(\lambda y-w)f\left(Q+\varepsilon\right)-f\left(Q\right)-\lambda^2 V(\lambda y-w)\varepsilon\\
&\hspace{50pt}-i\left(\frac{1}{\lambda}\frac{\partial \lambda}{\partial s}+b\right)\Lambda (Q+\varepsilon)+\left(1-\frac{\partial \gamma}{\partial s}\right)(Q+\varepsilon)+\left(\frac{\partial b}{\partial s}+b^2\right)\frac{|y|^2}{4}(Q+\varepsilon)\nonumber\\
&\hspace{70pt}-\left(\frac{1}{\lambda}\frac{\partial \lambda}{\partial s}+b\right)b\frac{|y|^2}{2}(Q+\varepsilon)+i\frac{1}{\lambda}\frac{\partial w}{\partial s}\cdot\nabla(Q+\varepsilon)+\frac{1}{2}\frac{b}{\lambda}\frac{\partial w}{\partial s}\cdot y(Q+\varepsilon).\nonumber
\end{align}
\end{lemma}

\begin{proof}
This result is proven via direct calculation.
\end{proof}

\begin{lemma}
\label{theorem:gTaylor}
For $g$,
\begin{align*}
g(\lambda y-w)&=1+\frac{1}{6}\sum_{|\alpha|=3}(\lambda y-w)^\alpha\frac{\partial^\alpha g}{\partial x^\alpha}(0)+\frac{1}{6}\sum_{|\alpha|=4}\int_0^1 (\lambda y-w)^\alpha\frac{\partial^\alpha g}{\partial x^\alpha}(\tau(\lambda y-w))(1-\tau)^3d\tau\\
\frac{\partial g}{\partial x_j}(\lambda y-w)&=\frac{1}{2}\sum_{|\alpha|=2}(\lambda y-w)\frac{\partial^\alpha g}{\partial x^\alpha}(0)+\frac{1}{2}\sum_{|\alpha|=3}\int_0^1 (\lambda y-w)^\alpha\frac{\partial^\alpha\partial g}{\partial x^\alpha\partial x_j}(\tau(\lambda y-w))(1-\tau)^2d\tau
\end{align*}
hold.
\end{lemma}

\begin{proof}
This result is proven via direct calculation by Taylor's theorem.
\end{proof}

\begin{lemma}
\label{imepsiesti}
For all $s\in(s_*,s_1]$,
\begin{align}
\label{nablaortho}
\left|(\im\varepsilon(s),\nabla Q)_2\right|\lesssim s^{-2}.
\end{align}
\end{lemma}

\begin{proof}
According to a direct calculation, we have
\[
\frac{d}{dt}\im\int_{\mathbb{R}^N}u(t,x)\nabla\overline{u}(t,x)dx=2\left(g|u|^\frac{4}{N}u-Vu(t),\nabla u(t)\right)_2=\left(-\frac{1}{1+\frac{2}{N}}\nabla g|u|^\frac{4}{N}u+\nabla Vu(t),u(t)\right)_2.
\]
Firstly, according to \eqref{V1growth}, we obtain
\begin{align*}
\left|\left(\nabla Vu(t),u(t)\right)_2\right|&=\left|\left((\nabla V)(\tilde{\lambda}(t)y-\tilde{w}(t))(Q+\tilde{\varepsilon}(t)),Q+\tilde{\varepsilon}(t)\right)_2\right|\\
&\lesssim \|Q+\varepsilon\|_2^2+\|(\tilde{\lambda}(t)y-w)(Q+\varepsilon)\|_2\|Q+\varepsilon\|_2\\
&\lesssim 1.
\end{align*}
Secondly, we obtain
\begin{align*}
&\frac{1}{2+\frac{4}{N}}\left(\nabla g|u|^\frac{4}{N}u+\nabla Vu(t),u(t)\right)_2=\frac{1}{\lambda^2}\int_{\mathbb{R}^N}(\nabla g)(\lambda y-w)F(Q+\varepsilon)dy\\
=&\frac{1}{\lambda^2}\int_{\mathbb{R}^N}(\nabla g)(\lambda y-w)\left(F(Q+\varepsilon)-F(Q)-dF(Q)(\varepsilon)\right)dy+\frac{1}{\lambda^2}\int_{\mathbb{R}^N}(\nabla g)(\lambda y-w)\left(F(Q)+dF(Q)(\varepsilon)\right)dy.
\end{align*}
Therefore, since $(\nabla g)(\lambda y-w)Q=O(\lambda^2+|w|^2)$, we obtain
\[
\left|\left(\nabla g|u|^\frac{4}{N}u+\nabla Vu(t),u(t)\right)_2\right|\lesssim \frac{1}{\lambda^2}\|\varepsilon\|_{H^1}^2+1+\frac{|w|^2}{\lambda^2}\lesssim 1
\]

Accordingly, we obtain
\begin{align*}
\left|\im\int_{\mathbb{R}^N}u(t(s),x)\nabla\overline{u}(t(s),x)dx\right|&\lesssim \int_s^{s_1}\lambda(\sigma)^2\left|\left(\nabla Vu(t(\sigma)),u(t(\sigma))\right)_2\right|d\sigma\\
&\lesssim\int_s^{s_1}\sigma^{-2}d\sigma\lesssim s^{-1}.
\end{align*}
Therefore, we obtain
\begin{align*}
&2(\im\varepsilon(s),\nabla Q)_2+(\varepsilon(s),i\nabla\varepsilon(s))_2+\frac{b}{2}\int_{\mathbb{R}^N}y\left|Q(y)+\varepsilon(s,y)\right|^2dy\\
=&\ \lambda\im\int_{\mathbb{R}^N}u(t(s),x)\nabla\overline{u}(t(s),x)dx\\
=&\ O\left(s^{-2}\right)
\end{align*}
Moreover, from \eqref{bootstrap} and the orthogonal conditions \eqref{orthocondi}, we obtain
\begin{align*}
2(\varepsilon(s),i\nabla\varepsilon(s))_2+b\int_{\mathbb{R}^N}y\left|Q(y)+\varepsilon(s,y)\right|^2dy&=2(\varepsilon(s),i\nabla\varepsilon(s))_2+b\int_{\mathbb{R}^N}y|\varepsilon(s,y)|^2dy\\
&=O(s^{-2L}).
\end{align*}
Consequently, we obtain \eqref{nablaortho}.
\end{proof}

\begin{lemma}[Estimation of modulation terms]
\label{modtermesti}
For all $s\in(s_*,s_1]$,
\begin{align}
\label{ortho}
2(\varepsilon(s),Q)_2&=-\left\|\varepsilon(s)\right\|_2^2,\\
\label{modesti}
\left|\Mod(s)\right|&\lesssim s^{-3},\\
\label{defbesti}
\left|\frac{\partial b}{\partial s}+b^2\right|&\lesssim s^{-2L}.
\end{align}
\end{lemma}

\begin{proof}
According to the mass conservation, we have
\[
2\left(\varepsilon,Q\right)_2=\left\|u\right\|_2^2-\left\|Q\right\|_2^2-\left\|\varepsilon\right\|_2^2=-\left\|\varepsilon\right\|_2^2,
\]
meaning \eqref{ortho} holds.

From Lemma \ref{theorem:epsieq},
\begin{align*}
i\frac{\partial \varepsilon}{\partial s}&=L_+\re\varepsilon+iL_-\im\varepsilon-\left(g(\lambda y-w)-1\right)f\left(Q+\varepsilon\right)-\left(f(Q+\varepsilon)-f(Q)-df(Q)(\varepsilon)\right)+\lambda^2 V(\lambda y-w)\varepsilon\\
&\hspace{50pt}+i\left(\frac{1}{\lambda}\frac{\partial \lambda}{\partial s}+b\right)\Lambda (Q+\varepsilon)-\left(1-\frac{\partial \gamma}{\partial s}\right)(Q+\varepsilon)-\left(\frac{\partial b}{\partial s}+b^2\right)\frac{|y|^2}{4}(Q+\varepsilon)\\
&\hspace{70pt}+\left(\frac{1}{\lambda}\frac{\partial \lambda}{\partial s}+b\right)b\frac{|y|^2}{2}(Q+\varepsilon)-i\frac{1}{\lambda}\frac{\partial w}{\partial s}\cdot\nabla(Q+\varepsilon)-\frac{1}{2}\frac{b}{\lambda}\frac{\partial w}{\partial s}\cdot y(Q+\varepsilon)+\Psi
\end{align*}
holds.

From the orthogonal properties \ref{orthocondi}, we have
\[
0=\frac{d}{ds}(i\varepsilon,\Lambda Q)_2=\frac{d}{ds}(i\varepsilon,i|\cdot|^2Q)_2=\frac{d}{ds}(i\varepsilon,\rho)_2,\quad 0=\frac{d}{ds}(i\varepsilon,iyQ)_2
\]

For $v=\Lambda Q$, $i|y|^2Q$, $\rho$, or $iy_jQ$, the following estimates hold:
\[
|f\left(Q+\varepsilon\right)-f\left(Q\right)-df(Q)(\varepsilon)||v|\lesssim |\varepsilon|^2,\quad|(\lambda^2 V(\lambda y-w)\varepsilon,v)_2|\lesssim s^{-3}\|\varepsilon\|_2.
\]

Firstly, we obtain
\begin{align*}
0&=\left(i\frac{\partial \varepsilon}{\partial s},\Lambda Q\right)_2\\
&=-2\left(\re\varepsilon,Q\right)_2-\left(f(Q+\varepsilon)-f(Q),(g(\lambda y-w)-1)\Lambda Q\right)_2-\left(f(Q),(g(\lambda y-w)-1)\Lambda Q\right)_2\\
&\hspace{20pt}+O\left(\|\varepsilon\|_{H^1}^2\right)+O\left(s^{-3}\|\varepsilon\|_{H^1}\right)+\frac{1}{4}\|yQ\|_2^2\left(\frac{\partial b}{\partial s}+b^2\right)+O\left(s\|\varepsilon\|_{H^1}|\Mod|\right)+O\left(\lambda^2|w|+\lambda^4\right)\\
&=O\left(\|\varepsilon\|_{H^1}^2\right)+O\left(s^{-3}\|\varepsilon\|_{H^1}\right)+O\left(\lambda^2|w|\right)+\frac{1}{4}\|yQ\|_2^2\left(\frac{\partial b}{\partial s}+b^2\right)+O\left(s\|\varepsilon\|_{H^1}|\Mod|\right).
\end{align*}
Therefore, we obtain
\[
\left|\frac{\partial b}{\partial s}+b^2\right|\lesssim s^{-2L}+s^{-\left(\frac{1}{2}+\frac{1}{K}\right)}|\Mod(s)|.
\]

Secondly, we obtain
\begin{align*}
0&=\left(i\frac{\partial \varepsilon}{\partial s},i|\cdot|^2Q\right)_2\\
&=-4\left(\im\varepsilon,\Lambda Q\right)_2-\left(f(Q+\varepsilon)-f(Q),(g(\lambda y-w)-1)i|\cdot|^2Q\right)_2-\left(f(Q),(g(\lambda y-w)-1)i|\cdot|^2 Q\right)_2\\
&\hspace{20pt}+O\left(\|\varepsilon\|_{H^1}^2\right)+O\left(s^{-3}\|\varepsilon\|_{H^1}\right)-\|yQ\|_2^2\left(\frac{1}{\lambda}\frac{\partial \lambda}{\partial s}+b\right)+O\left(s\|\varepsilon\|_{H^1}|\Mod|\right)\\
&=O\left(\|\varepsilon\|_{H^1}^2\right)+O\left(s^{-3}\|\varepsilon\|_{H^1}\right)-\|yQ\|_2^2\left(\frac{1}{\lambda}\frac{\partial \lambda}{\partial s}+b\right)+O\left(s\|\varepsilon\|_{H^1}|\Mod|\right).
\end{align*}
Therefore, we obtain
\begin{align}
\label{predefbesti}
\left|\frac{1}{\lambda}\frac{\partial \lambda}{\partial s}+b\right|\lesssim s^{-2L}+s^{-\left(\frac{1}{2}+\frac{1}{K}\right)}|\Mod(s)|.
\end{align}

Thirdly, we obtain
\begin{align*}
0&=\left(i\frac{\partial \varepsilon}{\partial s},\rho\right)_2\\
&=\left(\re\varepsilon,|\cdot|^2Q\right)_2-\left(f(Q+\varepsilon)-f(Q),(g(\lambda y-w)-1)\rho\right)_2-\left(f(Q),(g(\lambda y-w)-1)\rho\right)_2\\
&\hspace{20pt}+O\left(\|\varepsilon\|_{H^1}^2\right)+O\left(s^{-3}\|\varepsilon\|_{H^1}\right)-(Q,\rho)_2\left(1-\frac{\partial \gamma}{\partial s}\right)+O\left(s\|\varepsilon\|_{H^1}|\Mod|\right)+\left|\frac{\partial b}{\partial s}+b^2\right|+b\left|\frac{1}{\lambda}\frac{\partial \lambda}{\partial s}+b\right|+O\left(\lambda^2|w|+\lambda^4\right)\\
&=O\left(\|\varepsilon\|_{H^1}^2\right)+O\left(s^{-3}\|\varepsilon\|_{H^1}\right)+O\left(\lambda^2|w|\right)-(Q,\rho)_2\left(1-\frac{\partial \gamma}{\partial s}\right)+\left|\frac{\partial b}{\partial s}+b^2\right|+b\left|\frac{1}{\lambda}\frac{\partial \lambda}{\partial s}+b\right|+O\left(s\|\varepsilon\|_{H^1}|\Mod|\right).
\end{align*}
Therefore, we obtain
\[
\left|1-\frac{\partial\gamma}{\partial s}\right|\lesssim s^{-2L}+s^{-\left(\frac{1}{2}+\frac{1}{K}\right)}|\Mod(s)|.
\]

Fourthly, we obtain
\begin{align*}
0&=\left(i\frac{\partial \varepsilon}{\partial s},iy_jQ\right)_2\\
&=\left(\im\varepsilon,\frac{\partial Q}{\partial y_j}\right)_2-\left(f(Q+\varepsilon)-f(Q),(g(\lambda y-w)-1)iy_j Q\right)_2-\left(f(Q),(g(\lambda y-w)-1)iy_j Q\right)_2\\
&\hspace{20pt}+O\left(\|\varepsilon\|_{H^1}^2\right)+O\left(s^{-3}\|\varepsilon\|_{H^1}\right)-\frac{1}{\lambda}\frac{\partial w_j}{\partial s}\left(\frac{\partial Q}{\partial y_j},y_j Q\right)-\frac{1}{2}\frac{b}{\lambda}\frac{\partial w_j}{\partial s}\|y_jQ\|_2^2+O\left(s\|\varepsilon\|_{H^1}|\Mod|\right)\\
&=O(s^{-2})+O\left(\|\varepsilon\|_{H^1}^2\right)+O\left(s^{-3}\|\varepsilon\|_{H^1}\right)-\frac{1}{\lambda}\frac{\partial w_j}{\partial s}\left(\frac{\partial Q}{\partial y_j},y_j Q\right)-\frac{1}{2}\frac{b}{\lambda}\frac{\partial w_j}{\partial s}\|y_jQ\|_2^2+O\left(s\|\varepsilon\|_{H^1}|\Mod|\right).
\end{align*}
Therefore, we obtain
\[
\left|\frac{\partial w}{\partial s}\right|\lesssim s^{-3}+s^{-1}|\Mod(s)|.
\]

Accordingly, we obtain
\[
\left|\Mod(s)\right|\lesssim s^{-3}+s^{-1}\left|\Mod(s)\right|,
\]
so that \eqref{modesti} holds. Moreover, from \eqref{predefbesti}, we obtain \eqref{defbesti}.
\end{proof}

\section{Modified energy function}
\label{sec:MEF}
In this section, we proceed with a modified version of the technique presented in Le Coz, Martel, and Rapha\"{e}l \cite{LMR} and Rapha\"{e}l and Szeftel \cite{RSEU}. Let $m$, $\epsilon_1$, and $\epsilon_2$ satisfy
\[
1<1+\epsilon_1<\frac{m}{2}<L,\quad 0<\epsilon_2<\frac{m\mu\epsilon_1}{16},
\]
where $\mu$ is from the coercivity \eqref{Lcoer} of $L_+$ and $L_-$. Moreover, we define
\begin{align*}
H(s,\varepsilon)&:=\frac{1}{2}\left\|\varepsilon\right\|_{H^1}^2+\epsilon_2b^2\left\||y|\varepsilon\right\|_2^2-\int_{\mathbb{R}^N}g(\lambda y-w)\left(F(Q(y)+\varepsilon(y))-F(Q(y))-dF(Q(y))(\varepsilon(y))\right)dy\\
&\hspace{200pt}+\frac{1}{2}\lambda^2\int_{\mathbb{R}^N}V(\lambda y-w)|\varepsilon(y)|^2dy,\\
S(s,\varepsilon)&:=\frac{1}{\lambda^m}H(s,\varepsilon).
\end{align*}

\begin{lemma}[Coercivity of $H$]
\label{Hcoer}
For all $s\in(s_*,s_1]$, 
\[
H(s,\varepsilon)\geq \frac{\mu}{4}\|\varepsilon\|_{H^1}^2+\epsilon_2b^2\left\||y|\varepsilon\right\|_2^2.
\]
\end{lemma}

\begin{proof}
Firstly, we have
\begin{align*}
&\int_{\mathbb{R}^N}\left(F(Q(y)+\varepsilon(y))-F(Q(y))-dF(Q(y))(\varepsilon(y))-\frac{1}{2}d^2F(Q(y))(\varepsilon(y),\varepsilon(y))\right)dy\\
=&\ O\left(\|\varepsilon\|_{H^1}^3+\|\varepsilon\|_{H^1}^{2+\frac{4}{N}}\right).
\end{align*}

Secondly, according to \eqref{pote1'}, we have
\[
\left|\lambda^2\int_{\mathbb{R}^N}V(\lambda y-w)|\varepsilon(y)|^2dy\right|=o\left(\|\varepsilon\|_{H^1}^2\right).
\]

Thirdly, we have
\begin{align*}
\left|\left(g(\lambda y-w)-1\right)\left(F(Q(y)+\varepsilon(y))-F(Q(y))-dF(Q(y))(\varepsilon(y))\right)\right|&\lesssim\left|\left(g(\lambda y-w)-1\right)\right|(Q^\frac{4}{N}+|\varepsilon|^\frac{4}{N})|\varepsilon|^2\\
\lesssim (\lambda^3+|w|^3)|\varepsilon|^2+|\varepsilon|^{2+\frac{4}{N}}.
\end{align*}
Therefore, we obtain
\[
\int_{\mathbb{R}^N}\left)g(\lambda y-w)-1\right)\left(F(Q(y)+\varepsilon(y))-F(Q(y))-dF(Q(y))(\varepsilon(y))\right)dy=o(\|\varepsilon\|_{H^1}^2)
\]

Finally, since
\[
\left\|\varepsilon\right\|_{H^1}^2-\int_{\mathbb{R}^N}d^2F(Q(y))(\varepsilon(y),\varepsilon(y))dy=\left(L_+\re\varepsilon,\re\varepsilon\right)_2+\left( L_-\im\varepsilon,\im\varepsilon\right)_2,
\]
we obtain Lemma \ref{Hcoer}.
\end{proof}

From Lemma \ref{Hcoer} and the definition of $S$, we obtain the following:

\begin{corollary}[Estimation of $S$]
\label{Sesti}
For all $s\in(s_*,s_1]$, 
\[
\frac{1}{\lambda^m}\left(\frac{\mu}{4}\|\varepsilon\|_{H^1}^2+\epsilon_2b^2\left\||y|\varepsilon\right\|_2^2\right)\leq S(s,\varepsilon)\lesssim\frac{1}{\lambda^m}\left(\|\varepsilon\|_{H^1}^2+b^2\left\||y|\varepsilon\right\|_2^2\right).
\]
\end{corollary}

\begin{lemma}
\label{Lambda}
For all $s\in(s_*,s_1]$, 
\begin{align}
\label{flambdaesti}
\left|\left((g(\lambda y-w)-1)f(Q+\varepsilon)-f(Q),\Lambda \varepsilon\right)_2\right|&\lesssim \|\varepsilon\|_{H^1}^2,\\
\label{fdefesti}
\left|\left((g(\lambda y-w)-1)f(Q+\varepsilon)-f(Q),\nabla \varepsilon\right)_2\right|&\lesssim \|\varepsilon\|_{H^1}^2,\\
\label{Vlambdaesti}
\left|\lambda^2\left( V(\lambda y-w)\varepsilon,\Lambda\varepsilon\right)_2\right|&\lesssim s^{-2}\left(\|\varepsilon\|_{H^1}^2+b^2\||y|\varepsilon\|_2^2\right),\\
\label{Vdefesti}
\left|\lambda^2\left( V(\lambda y-w)\varepsilon,\nabla\varepsilon\right)_2\right|&\lesssim s^{-3}\left(\|\varepsilon\|_{H^1}^2+b^2\||y|\varepsilon\|_2^2\right).
\end{align}
\end{lemma}

\begin{proof}
%For \eqref{flambdaesti}, see \cite[Section 5.4]{LMR}.

Firstly,
\begin{align*}
&\nabla \left(g(\lambda y-w)\left(F(Q+\varepsilon)-F(Q)-dF(Q)(\varepsilon)\right)\right)\\
=&\lambda(\nabla g)(\lambda y-w)\left(F(Q+\varepsilon)-F(Q)-dF(Q)(\varepsilon)\right)\\
&\hspace{20pt}+g(\lambda y-w)\re\left(f(Q+\varepsilon)\nabla(Q+\varepsilon)-f(Q)\nabla Q-df(Q)(\varepsilon)\nabla Q-f(Q)\nabla\overline{\varepsilon}\right)\\
=&\lambda(\nabla g)(\lambda y-w)\left(F(Q+\varepsilon)-F(Q)-dF(Q)(\varepsilon)\right)+g(\lambda y-w)\re\left(f(Q+\varepsilon)-f(Q)-df(Q)(\varepsilon)\right)\nabla Q\\
&\hspace{20pt}-g(\lambda y-w)\re\left(\left(f(Q+\varepsilon)-f(Q)\right)\nabla\overline{\varepsilon}\right)
\end{align*}
Therefore, we obtain
\begin{align*}
&g(\lambda y-w)\re\left(\left(f(Q+\varepsilon)-f(Q)\right)\Lambda\overline{\varepsilon}\right)\\
=&\frac{N}{2}g(\lambda y-w)\re\left(\left(f(Q+\varepsilon)-f(Q)\right)\overline{\varepsilon}\right)+y\cdot\nabla\left(g(\lambda y-w)\left(F(Q+\varepsilon)-F(Q)-dF(Q)(\varepsilon)\right)\right)\\
&\hspace{20pt}-w\cdot(\nabla g)(\lambda y-w)\left(F(Q+\varepsilon)-F(Q)-dF(Q)(\varepsilon)\right)-(\lambda y-w)\cdot(\nabla g)(\lambda y-w)\left(F(Q+\varepsilon)-F(Q)-dF(Q)(\varepsilon)\right)\\
&\hspace{40pt}g(\lambda y-w)\re\left(f(Q+\varepsilon)-f(Q)-df(Q)(\varepsilon)\right)y\cdot\nabla Q
\end{align*}
and
\begin{align*}
&g(\lambda y-w)\re\left(\left(f(Q+\varepsilon)-f(Q)\right)\nabla\overline{\varepsilon}\right)\\
=&\nabla\left(g(\lambda y-w)\left(F(Q+\varepsilon)-F(Q)-dF(Q)(\varepsilon)\right)\right)-\lambda(\nabla g)(\lambda y-w)\left(F(Q+\varepsilon)-F(Q)-dF(Q)(\varepsilon)\right)\\
&\hspace{40pt}g(\lambda y-w)\re\left(f(Q+\varepsilon)-f(Q)-df(Q)(\varepsilon)\right)\nabla Q.
\end{align*}
Therefore, we obtain
\begin{align*}
\left((g(\lambda y-w)-1)f(Q+\varepsilon)-f(Q),\Lambda \varepsilon\right)_2&=O(\|\varepsilon\|_{H^1}^2),\\
\left((g(\lambda y-w)-1)f(Q+\varepsilon)-f(Q),\nabla \varepsilon\right)_2&=O(\|\varepsilon\|_{H^1}^2),
\end{align*}
so that \eqref{flambdaesti} and \eqref{fdefesti} hold.

For \eqref{Vlambdaesti}, from \eqref{V1growth}, a direct calculation shows
\[
\left(V(\lambda y-w)\varepsilon,\Lambda\varepsilon\right)_2=-\frac{1}{2}\left(\lambda y\cdot(\nabla V)(\lambda y-w)\varepsilon,\varepsilon\right)_2=O\left(\lambda\|y\varepsilon\|_2\|(1+|\lambda y-w|)\varepsilon\|_2\right).
\]
Therefore, we obtain \eqref{Vlambdaesti}.

For \eqref{Vdefesti}, from \eqref{V1growth}, a direct calculation shows
\[
\left(V(\lambda y-w)\varepsilon,\nabla\varepsilon\right)_2=-\frac{1}{2}\left(\lambda (\nabla V)(\lambda y-w)\varepsilon,\varepsilon\right)_2=O\left(\lambda\|\varepsilon\|_2\|(1+|\lambda y-w|)\varepsilon\|_2\right).
\]
Therefore, we obtain \eqref{Vdefesti}.
\end{proof}

We define $\kappa$ by
\[
\kappa:=\frac{1}{4}-\frac{1}{2K}=\frac{2-L}{2}.
\]

\begin{lemma}[Derivative of $H$ in time]
\label{Hdef}
For all $s\in(s_*,s_1]$, 
\[
\frac{d}{ds}H(s,\varepsilon(s))\geq -b\left(\frac{4\epsilon_2}{\epsilon_1}\|\varepsilon\|_{H^1}^2+\left(\frac{m}{2}+1+\epsilon_1\right)\epsilon_2b^2\left\||y|\varepsilon\right\|_2^2+C's^{-4}\right).
\]
\end{lemma}

\begin{proof}
Firstly, we have
\[
\frac{d}{ds}H(s,\varepsilon(s))=\frac{\partial H}{\partial s}(s,\varepsilon(s))+\left\langle i\frac{\partial H}{\partial \varepsilon},i\frac{\partial \varepsilon}{\partial s}\right\rangle.
\]
Secondly, we have
\begin{align*}
\frac{\partial H}{\partial \varepsilon}&=-\Delta \varepsilon+\varepsilon+2\epsilon_2b^2|y|^2\varepsilon-g(\lambda y-w)(f(Q+\varepsilon)-f(Q))+\lambda^2 V(\lambda y-w)\varepsilon\\
&=L_+\re\varepsilon+iL_-\im\varepsilon+2\epsilon_2b^2|y|^2\varepsilon-\left(g(\lambda y-w)-1\right)df(Q)(\varepsilon)\\
&\hspace{20pt}-g(\lambda y-w)(f(Q+\varepsilon)-f(Q)-df(Q)(\varepsilon))+\lambda^2 V(\lambda y-w)\varepsilon,\\
i\frac{\partial \varepsilon}{\partial s}&=\frac{\partial H}{\partial \varepsilon}-2\epsilon_2b^2|y|^2\varepsilon-(g(\lambda y-w)-1)f(Q)+\ModOp(Q+\varepsilon)+\Psi,
\end{align*}
where
\[
\ModOp v:=i\left(\frac{1}{\lambda}\frac{\partial \lambda}{\partial s}+b\right)\Lambda v-\left(1-\frac{\partial \gamma}{\partial s}\right)v-\left(\frac{\partial b}{\partial s}+b^2\right)\frac{|y|^2}{4}v+\left(\frac{1}{\lambda}\frac{\partial \lambda}{\partial s}+b\right)b\frac{|y|^2}{2}c-i\frac{1}{\lambda}\frac{\partial w}{\partial s}\cdot\nabla v-\frac{1}{2}\frac{b}{\lambda}\frac{\partial w}{\partial s}\cdot yv.
\]

For $\frac{\partial H}{\partial s}$, we have
\begin{align*}
\frac{\partial H}{\partial s}&=2\epsilon_2b\frac{\partial b}{\partial s}\|y\varepsilon\|_2^2-\int_{\mathbb{R}^N}\left(\frac{\partial \lambda}{\partial s}y-\frac{\partial w}{\partial s}\right)\cdot(\nabla g)(\lambda y-w)\left(F(Q+\varepsilon)-F(Q)-dF(Q)(\varepsilon)\right)dy\\
&\hspace{20pt}+\lambda^2\frac{1}{\lambda}\frac{\partial\lambda}{\partial s}\int_{\mathbb{R}^N}V(\lambda y-w)|\varepsilon|^2dy+\frac{1}{2}\lambda^2\int_{\mathbb{R}^N}\left(\frac{\partial \lambda}{\partial s}y-\frac{\partial w}{\partial s}\right)\cdot(\nabla V)(\lambda y-w)|\varepsilon|^2dy.
\end{align*}
Moreover, since $\frac{1}{\lambda}\frac{\partial \lambda}{\partial s}\approx -b$, $\frac{\partial b}{\partial s}\approx -b^2$, and $\lambda\approx b$, we have
\begin{align*}
2\epsilon_2b\frac{\partial b}{\partial s}\|y\varepsilon\|_2^2&\geq -2(1+\epsilon)\epsilon_2b^3\|y\varepsilon\|_2^2,\\
\left|\lambda^2\int_{\mathbb{R}^N}V(\lambda y-w)|\varepsilon|^2dy\right|&\lesssim \lambda^{2-\frac{N}{p}}\|\varepsilon\|_{H^1}^2,\\
\left|\int_{\mathbb{R}^N}\lambda y\cdot(\nabla V)(\lambda y-w)|\varepsilon|^2dy\right|&\lesssim b\|y\varepsilon\|_2\|(1+\lambda|y|+|w|)\varepsilon\|_2\lesssim \|\varepsilon\|_{H^1}^2+b^2\|y\varepsilon\|_2^2,\\
\left|\int_{\mathbb{R}^N}(\nabla V)(\lambda y-w)|\varepsilon|^2dy\right|&\lesssim \|\varepsilon\|_2\|(1+\lambda|y|+|w|)\varepsilon\|_2\lesssim \|\varepsilon\|_{H^1}^2+b^2\|y\varepsilon\|_2^2,\\
\left|\int_{\mathbb{R}^N}(\nabla g)(\lambda y-w)\left(F(Q+\varepsilon)-F(Q)-dF(Q)(\varepsilon)\right)dy\right|&\lesssim \|\varepsilon\|_{H^1}^2,\\
\left|\int_{\mathbb{R}^N}\left(\lambda y-w\right)\cdot(\nabla g)(\lambda y-w)\left(F(Q+\varepsilon)-F(Q)-dF(Q)(\varepsilon)\right)dy\right|&\lesssim\int_{\mathbb{R}^N}\left|\left(\lambda y-w\right)\cdot(\nabla g)(\lambda y-w)\right|(Q^\frac{4}{N}+|\varepsilon|^\frac{4}{N})|\varepsilon|^2dy\\
&\lesssim (\lambda+|w|)\|\varepsilon\|_2^2+\|\varepsilon\|_{H^1}^{2+\frac{4}{N}}.
\end{align*}
Therefore,
\begin{align}
\label{esti-1}
\frac{\partial H}{\partial s}\geq -2(1+\epsilon)\epsilon_2b^3\|y\varepsilon\|_2^2+o\left(b\left(\|\varepsilon\|_{H^1}^2+b^2\|y\varepsilon\|_2^2\right)\right).
\end{align}

For $\left\langle i\frac{\partial H}{\partial \varepsilon},2\epsilon_2b^2|y|^2\varepsilon\right\rangle$, we have
\begin{align*}
&\left\langle \frac{\partial H}{\partial \varepsilon},2i\epsilon_2b^2|y|^2\varepsilon\right\rangle\\
=&\left\langle -\Delta \varepsilon+\varepsilon+2\epsilon_2b^2|y|^2\varepsilon-g(\lambda y-w)(f(Q+\varepsilon)-f(Q))+\lambda^2 V(\lambda y-w)\varepsilon,2i\epsilon_2b^2|y|^2\varepsilon\right\rangle\\
=&4\epsilon_2b^2(\nabla \varepsilon, iy\varepsilon)_2-2\epsilon_2b^2\left(g(\lambda y-w)\left(|Q+\varepsilon|^\frac{4}{N}-Q^\frac{4}{N}\right)Q,|y|^2\varepsilon\right)_2.
\end{align*}
Therefore, we obtain
\begin{align}
\label{esti-2}
\left|\left\langle i\frac{\partial H}{\partial \varepsilon},2\epsilon_2b^2|y|^2\varepsilon\right\rangle\right|\leq 4\epsilon_2b^2\|\varepsilon\|_{H^1}\|y\varepsilon\|_2+o(b\|\varepsilon\|_{H^1}^2).
\end{align}

For $\left\langle i\frac{\partial H}{\partial \varepsilon},(g(\lambda y-w)-1)f(Q)\right\rangle$, we have
\begin{align*}
&\left\langle \frac{\partial H}{\partial \varepsilon},i(g(\lambda y-w)-1)f(Q)\right\rangle\\
=&\left\langle -\Delta \varepsilon+\varepsilon+2\epsilon_2b^2|y|^2\varepsilon-g(\lambda y-w)(f(Q+\varepsilon)-f(Q))+\lambda^2 V(\lambda y-w)\varepsilon,i(g(\lambda y-w)-1)f(Q)\right\rangle\\
=&(\nabla\varepsilon,i\lambda(\nabla g)(\lambda y-w)f(Q)+i(g(\lambda y-w)-1)df(Q)\nabla Q)_2+O(s^{-3}\|\varepsilon\|_{H^1})\\
=&O(s^{-3}\|\varepsilon\|_{H^1}).
\end{align*}
Therefore, we obtain
\begin{align}
\label{esti-3}
\left|\left\langle i\frac{\partial H}{\partial \varepsilon},(g(\lambda y-w)-1)f(Q)\right\rangle\right|\leq \epsilon_3b\|\varepsilon\|_{H^1}^2+\frac{C'}{\epsilon_3}s^{-5}.
\end{align}

For $\left\langle i\frac{\partial H}{\partial \varepsilon},\Psi\right\rangle$, we have
\begin{align*}
&\left\langle \frac{\partial H}{\partial \varepsilon},i\Psi\right\rangle\\
=&\left\langle -\Delta \varepsilon+\varepsilon+2\epsilon_2b^2|y|^2\varepsilon-g(\lambda y-w)(f(Q+\varepsilon)-f(Q))+\lambda^2 V(\lambda y-w)\varepsilon,i\Psi\right\rangle\\
=&O(s^{-3}\|\varepsilon\|_{H^1}).
\end{align*}
Therefore, we obtain
\begin{align}
\label{esti-4}
\left|\left\langle i\frac{\partial H}{\partial \varepsilon},\Psi\right\rangle\right|\leq \epsilon_3b\|\varepsilon\|_{H^1}^2+\frac{C'}{\epsilon_3}s^{-5}.
\end{align}

Next, we consider $\left\langle i\frac{\partial H}{\partial \varepsilon},\ModOp Q\right\rangle$. Firstly, for $\left\langle i\frac{\partial H}{\partial \varepsilon},i \Lambda Q\right\rangle$, we have
\begin{align*}
&\left\langle \frac{\partial H}{\partial \varepsilon},\Lambda Q\right\rangle\\
=&\left\langle L_+\re\varepsilon+iL_-\im\varepsilon+2\epsilon_2b^2|y|^2\varepsilon-\left(g(\lambda y-w)-1\right)df(Q)(\varepsilon)\right.\\
&\left.\hspace{20pt}-g(\lambda y-w)(f(Q+\varepsilon)-f(Q)-df(Q)(\varepsilon))+\lambda^2 V(\lambda y-w)\varepsilon,\Lambda Q\right\rangle\\
=&-2(\re\varepsilon,Q)_2+O(s^{-2}\|\varepsilon\|_{H^1})+O(s^{-3}\|\varepsilon\|_{H^1})+O(\|\varepsilon\|_{H^1}^2)+O(s^{-3}\|\varepsilon\|_{H^1}).
\end{align*}
Therefore, we obtain
\[
\left|\left\langle i\frac{\partial H}{\partial \varepsilon},i \Lambda Q\right\rangle\right|\lesssim \|\varepsilon\|_{H^1}^2+s^{-4}.
\]

Secondly, for $\left\langle i\frac{\partial H}{\partial \varepsilon},Q\right\rangle$, we have
\begin{align*}
&\left\langle \frac{\partial H}{\partial \varepsilon},iQ\right\rangle\\
=&\left\langle L_+\re\varepsilon+iL_-\im\varepsilon+2\epsilon_2b^2|y|^2\varepsilon-\left(g(\lambda y-w)-1\right)df(Q)(\varepsilon)\right.\\
&\left.\hspace{20pt}-g(\lambda y-w)(f(Q+\varepsilon)-f(Q)-df(Q)(\varepsilon))+\lambda^2 V(\lambda y-w)\varepsilon,iQ\right\rangle\\
=&O(s^{-2}\|\varepsilon\|_{H^1})+O(s^{-3}\|\varepsilon\|_{H^1})+O(\|\varepsilon\|_{H^1}^2)+O(s^{-3}\|\varepsilon\|_{H^1}).
\end{align*}
Therefore, we obtain
\[
\left|\left\langle i\frac{\partial H}{\partial \varepsilon},iQ\right\rangle\right|\lesssim \|\varepsilon\|_{H^1}^2+s^{-4}.
\]

Thirdly, for $\left\langle i\frac{\partial H}{\partial \varepsilon},|y|^2Q\right\rangle$, we have
\begin{align*}
&\left\langle \frac{\partial H}{\partial \varepsilon},i|y|^2Q\right\rangle\\
=&\left\langle L_+\re\varepsilon+iL_-\im\varepsilon+2\epsilon_2b^2|y|^2\varepsilon-\left(g(\lambda y-w)-1\right)df(Q)(\varepsilon)\right.\\
&\left.\hspace{20pt}-g(\lambda y-w)(f(Q+\varepsilon)-f(Q)-df(Q)(\varepsilon))+\lambda^2 V(\lambda y-w)\varepsilon,i|y|^2Q\right\rangle\\
=&-4(\im\varepsilon,\Lambda Q)_2+O(s^{-2}\|\varepsilon\|_{H^1})+O(s^{-3}\|\varepsilon\|_{H^1})+O(\|\varepsilon\|_{H^1}^2)+O(s^{-3}\|\varepsilon\|_{H^1}).
\end{align*}
Therefore, we obtain
\[
\left|\left\langle i\frac{\partial H}{\partial \varepsilon},|y|^2Q\right\rangle\right|\lesssim \|\varepsilon\|_{H^1}^2+s^{-4}.
\]

Fourthly, for $\left\langle i\frac{\partial H}{\partial \varepsilon},i\nabla Q\right\rangle$, we have
\begin{align*}
&\left\langle \frac{\partial H}{\partial \varepsilon},\frac{\partial Q}{\partial y_j}\right\rangle\\
=&\left\langle L_+\re\varepsilon+iL_-\im\varepsilon+2\epsilon_2b^2|y|^2\varepsilon-\left(g(\lambda y-w)-1\right)df(Q)(\varepsilon)\right.\\
&\left.\hspace{20pt}-g(\lambda y-w)(f(Q+\varepsilon)-f(Q)-df(Q)(\varepsilon))+\lambda^2 V(\lambda y-w)\varepsilon,\frac{\partial Q}{\partial y_j}\right\rangle\\
=&O(s^{-2}\|\varepsilon\|_{H^1})+O(s^{-3}\|\varepsilon\|_{H^1})+O(\|\varepsilon\|_{H^1}^2)+O(s^{-3}\|\varepsilon\|_{H^1}).
\end{align*}
Therefore, we obtain
\[
\left|\left\langle i\frac{\partial H}{\partial \varepsilon},i\nabla Q\right\rangle\right|\lesssim \|\varepsilon\|_{H^1}^2+s^{-4}.
\]

Fifthly, for $\left\langle i\frac{\partial H}{\partial \varepsilon},yQ\right\rangle$, we have
\begin{align*}
&\left\langle \frac{\partial H}{\partial \varepsilon},iy_jQ\right\rangle\\
=&\left\langle L_+\re\varepsilon+iL_-\im\varepsilon+2\epsilon_2b^2|y|^2\varepsilon-\left(g(\lambda y-w)-1\right)df(Q)(\varepsilon)\right.\\
&\left.\hspace{20pt}-g(\lambda y-w)(f(Q+\varepsilon)-f(Q)-df(Q)(\varepsilon))+\lambda^2 V(\lambda y-w)\varepsilon,iy_jQ\right\rangle\\
=&-\left(\im\varepsilon,\frac{\partial Q}{\partial y_j}\right)_2+O(s^{-2}\|\varepsilon\|_{H^1})+O(s^{-3}\|\varepsilon\|_{H^1})+O(\|\varepsilon\|_{H^1}^2)+O(s^{-3}\|\varepsilon\|_{H^1}).
\end{align*}
Therefore, we obtain
\[
\left|\left\langle i\frac{\partial H}{\partial \varepsilon},yQ\right\rangle\right|\lesssim \|\varepsilon\|_{H^1}^2+s^{-2}.
\]

Accordingly, we obtain
\begin{align}
\label{esti-5}
\left|\left\langle i\frac{\partial H}{\partial \varepsilon},\ModOp Q\right\rangle\right|&\lesssim s^{-3}\left(\|\varepsilon\|_{H^1}^2+s^{-4}\right)+s^{-2}\left(\|\varepsilon\|_{H^1}^2+s^{-4}\right)+s^{-3}\left(\|\varepsilon\|_{H^1}^2+s^{-2}\right).
\end{align}

Finally, we consider $\left\langle i\frac{\partial H}{\partial \varepsilon},\ModOp \varepsilon\right\rangle$. Firstly, for $\left\langle i\frac{\partial H}{\partial \varepsilon},i \Lambda \varepsilon\right\rangle$, we have
\begin{align*}
&\left\langle \frac{\partial H}{\partial \varepsilon},\Lambda \varepsilon\right\rangle\\
=&\left\langle -\Delta \varepsilon+\varepsilon+2\epsilon_2b^2|y|^2\varepsilon-g(\lambda y-w)(f(Q+\varepsilon)-f(Q))+\lambda^2 V(\lambda y-w)\varepsilon,\Lambda \varepsilon\right\rangle\\
=&O(\|\varepsilon\|_{H^1}^2)+O(b^2\|y\varepsilon\|_2^2)+O(\|\varepsilon\|_{H^1}^2).
\end{align*}
Therefore, we obtain
\[
\left|\left\langle i\frac{\partial H}{\partial \varepsilon},i \Lambda \varepsilon\right\rangle\right|\lesssim \|\varepsilon\|_{H^1}^2.
\]

Secondly, for $\left\langle i\frac{\partial H}{\partial \varepsilon},\varepsilon\right\rangle$, we have
\begin{align*}
&\left\langle \frac{\partial H}{\partial \varepsilon},i\varepsilon\right\rangle\\
=&\left\langle -\Delta \varepsilon+\varepsilon+2\epsilon_2b^2|y|^2\varepsilon-g(\lambda y-w)(f(Q+\varepsilon)-f(Q))+\lambda^2 V(\lambda y-w)\varepsilon,i\varepsilon\right\rangle\\
=&-\left(g(\lambda y-w)\left(|Q+\varepsilon|^\frac{4}{N}-Q^\frac{4}{N}\right)Q,i\varepsilon\right)_2=O(\|\varepsilon\|_{H^1}^2).
\end{align*}
Therefore, we obtain
\[
\left|\left\langle i\frac{\partial H}{\partial \varepsilon},\varepsilon\right\rangle\right|\lesssim \|\varepsilon\|_{H^1}^2.
\]

Thirdly, for $\left\langle i\frac{\partial H}{\partial \varepsilon},|y|^2\varepsilon\right\rangle$, we have
\begin{align*}
&\left\langle \frac{\partial H}{\partial \varepsilon},i|y|^2\varepsilon\right\rangle\\
=&\left\langle -\Delta \varepsilon+\varepsilon+2\epsilon_2b^2|y|^2\varepsilon-g(\lambda y-w)(f(Q+\varepsilon)-f(Q))+\lambda^2 V(\lambda y-w)\varepsilon,i|y|^2\varepsilon\right\rangle\\
=&(\nabla\varepsilon,iy\varepsilon)_2-\left(g(\lambda y-w)\left(|Q+\varepsilon|^\frac{4}{N}-Q^\frac{4}{N}\right)Q,i|y|^2\varepsilon\right)_2=O(\|\varepsilon\|_{H^1}\|y\varepsilon\|_2).
\end{align*}
Therefore, we obtain
\[
\left|\left\langle i\frac{\partial H}{\partial \varepsilon},|y|^2\varepsilon\right\rangle\right|\lesssim s\left(\|\varepsilon\|_{H^1}^2+b^2\|y\varepsilon\|_2^2\right).
\]

Fourthly, for $\left\langle i\frac{\partial H}{\partial \varepsilon},i\nabla\varepsilon\right\rangle$, we have
\begin{align*}
&\left\langle \frac{\partial H}{\partial \varepsilon},\nabla\varepsilon\right\rangle\\
=&\left\langle -\Delta \varepsilon+\varepsilon+2\epsilon_2b^2|y|^2\varepsilon-g(\lambda y-w)(f(Q+\varepsilon)-f(Q))+\lambda^2 V(\lambda y-w)\varepsilon,\nabla\varepsilon\right\rangle\\
=&-4\epsilon_2b^2(y\varepsilon,\varepsilon)_2+\left(-g(\lambda y-w)(f(Q+\varepsilon)-f(Q))+\lambda^2 V(\lambda y-w)\varepsilon,\nabla\varepsilon\right)_2=O(\|\varepsilon\|_{H^1}^2+b^2\|y\varepsilon\|_2^2).
\end{align*}
Therefore, we obtain
\[
\left|\left\langle i\frac{\partial H}{\partial \varepsilon},i\nabla \varepsilon\right\rangle\right|\lesssim \|\varepsilon\|_{H^1}^2+b^2\|y\varepsilon\|_2^2.
\]

Fifthly, for $\left\langle i\frac{\partial H}{\partial \varepsilon},y\varepsilon\right\rangle$, we have
\begin{align*}
&\left\langle \frac{\partial H}{\partial \varepsilon},iy_j\varepsilon\right\rangle\\
=&\left\langle -\Delta \varepsilon+\varepsilon+2\epsilon_2b^2|y|^2\varepsilon-g(\lambda y-w)(f(Q+\varepsilon)-f(Q))+\lambda^2 V(\lambda y-w)\varepsilon,iy_j\varepsilon\right\rangle\\
=&\left(\frac{\partial\varepsilon}{\partial y_j},i\varepsilon\right)_2-\left(g(\lambda y-w)\left(|Q+\varepsilon|^\frac{4}{N}-Q^\frac{4}{N}\right)Q,iy_j\varepsilon\right)_2=O\left(\|\varepsilon\|_{H^1}^2\right).
\end{align*}
Therefore, we obtain
\[
\left|\left\langle i\frac{\partial H}{\partial \varepsilon},y\varepsilon\right\rangle\right|\lesssim \|\varepsilon\|_{H^1}^2.
\]

Accordingly, we obtain
\begin{align}
\label{esti-6}
\left|\left\langle i\frac{\partial H}{\partial \varepsilon},\ModOp \varepsilon\right\rangle\right|&\lesssim s^{-3}\left(\|\varepsilon\|_{H^1}^2+b^2\|y\varepsilon\|_2^2\right)+s^{-2}\left(\|\varepsilon\|_{H^1}^2+b^2\|y\varepsilon\|_2^2\right).
\end{align}

Combining inequalities \eqref{esti-1}, \eqref{esti-2}, \eqref{esti-3}, \eqref{esti-4}, \eqref{esti-5}, and \eqref{esti-6}, we obtain
\begin{align*}
&\frac{d}{ds}H(s,\varepsilon(s))=\frac{\partial H}{\partial s}(s,\varepsilon(s))+\left\langle i\frac{\partial H}{\partial \varepsilon},i\frac{\partial \varepsilon}{\partial s}\right\rangle\\
\geq&-2(1+\epsilon)\epsilon_2b^3\|y\varepsilon\|_2^2+o\left(b\left(\|\varepsilon\|_{H^1}^2+b^2\|y\varepsilon\|_2^2\right)\right)-4\epsilon_2b^2\|\varepsilon\|_{H^1}\|y\varepsilon\|_2+o(b\|\varepsilon\|_{H^1}^2)\\
&\hspace{20pt}-2\epsilon_3b\|\varepsilon\|_{H^1}^2-\frac{C'}{\epsilon_3}s^{-5}+o\left(b\left(\|\varepsilon\|_{H^1}^2+b^2\|y\varepsilon\|_2^2\right)\right)-C''s^{-5}\\
\geq&-b\left(2\left(\frac{\epsilon_2}{\epsilon_1}+\epsilon_3+\epsilon\right)\|\varepsilon\|_{H^1}^2+2(1+\epsilon_1+\epsilon)\epsilon_2b^2\|y\varepsilon\|_2^2-\frac{C'}{\epsilon_3}s^{-4}\right)\\
\geq&-b\left(\frac{4\epsilon_2}{\epsilon_1}\|\varepsilon\|_{H^1}^2+\left(\frac{m}{2}+1+\epsilon_1\right)\epsilon_2b^2\|y\varepsilon\|_2^2-C's^{-4}\right),
\end{align*}
so that we obtain Lemma \ref{Hdef}.
\end{proof}

\begin{lemma}[Derivative of $S$ in time]
\label{Sdef}
For all $s\in(s_*,s_1]$,
\[
\frac{d}{ds}S(s,\varepsilon(s))\gtrsim \frac{b}{\lambda^m}\left(\|\varepsilon\|_{H^1}^2+b^2\left\||y|\varepsilon\right\|_2^2-s^{-(2L+\kappa)}\right).
\]
\end{lemma}

\begin{proof}
According to \eqref{modesti}, Lemma \ref{Lambda}, and Lemma \ref{Hdef}, we have
\begin{align*}
&\frac{d}{ds}S(s,\varepsilon(s))=m\frac{b}{\lambda^m}H(s,\varepsilon(s))-m\frac{1}{\lambda^m}\left(\frac{1}{\lambda}\frac{\partial \lambda}{\partial s}+b\right)H(s,\varepsilon(s))+\frac{1}{\lambda^m}\frac{d}{ds}H(s,\varepsilon(s))\\
\geq& \frac{b}{\lambda^m}\left(\left(\frac{m\mu}{4}-\frac{4\epsilon_2}{\epsilon_1}\right)\|\varepsilon\|_{H^1}^2+\left(\frac{m}{2}-(1+\epsilon_1)\right)\epsilon_2b^2\left\||y|\varepsilon\right\|_2^2-C's^{-4}\right).
\end{align*}
Therefore, we obtain Lemma \ref{Sdef}.
\end{proof}

\section{Bootstrap}
\label{sec:bootstrap}
In this section, we establish the estimates of the decomposition parameters by using a bootstrap argument and the estimates obtained in Section \ref{sec:MEF}.

\begin{lemma}
\label{rebootstrap}
There exists a sufficiently small $\epsilon_3>0$ such that for all $s\in(s_*,s_1]$, 
\begin{align}
\label{reepsiesti}
\left\|\varepsilon(s)\right\|_{H^1}^2+b(s)^2\left\||y|\varepsilon(s)\right\|_2^2&\lesssim s^{-(2L+\kappa)},\\
\label{relamesti}
\left|s\lambda(s)-1\right|&<(1-\epsilon_3)s^{-M},\\
\label{rebesti}
\left|sb(s)-1\right|&<(1-\epsilon_3)s^{-M},\\
\label{rewesti}
\left|w(s)\right|&\lesssim s^{-2}.
\end{align}
\end{lemma}

\begin{proof}
%By using Corollary \ref{Sesti} and Lemma \ref{Sdef} as in the proof of Lemma 6.1 in \cite{LMR}, we see that \eqref{reepsiesti}.

Let $C_\dagger$ be sufficiently large and $s_\dagger$ defined by
\[
s_\dagger:=\inf\left\{\ \sigma\in[s_*,s_1]\ \middle|\ \|\varepsilon(\tau)\|_{H^1}^2+b(\tau)^2\|y\varepsilon(\tau)\|_2^2\leq C_\dagger \tau^{-(2L+\kappa)}\quad(\tau\in[\sigma,s_1])\right\}.
\]
Assume $s_\dagger>s_*$. Then
\[
\|\varepsilon(s_\dagger)\|_{H^1}^2+b(s_\dagger)^2\|y\varepsilon(s_\dagger)\|_2^2=C_\dagger s_\dagger^{-(2L+\kappa)}
\]
holds. Moreover, let $s_\ddagger$ defined by
\[
s_\ddagger:=\sup\left\{\ \sigma\in[s_\dagger,s_1]\ \middle|\ \|\varepsilon(\tau)\|_{H^1}^2+b(\tau)^2\|y\varepsilon(\tau)\|_2^2\geq \tau^{-(2L+\kappa)}\quad(\tau\in[s_\dagger,\sigma])\right\}.
\]
Then, since $s_\ddagger<s_1$,
\[
\|\varepsilon(s_\ddagger)\|_{H^1}^2+b(s_\ddagger)^2\|y\varepsilon(s_\ddagger)\|_2^2=s_\ddagger^{-(2L+\kappa)}
\]
holds.

From Corollary \ref{Sesti} and Lemma \ref{Sdef},
\begin{align*}
\frac{C_1}{\lambda^m}\left(\frac{\mu}{4}\|\varepsilon\|_{H^1}^2+b^2\left\||y|\varepsilon\right\|_2^2\right)&\leq S(s,\varepsilon)\leq\frac{C_2}{\lambda^m}\left(\|\varepsilon\|_{H^1}^2+b^2\left\||y|\varepsilon\right\|_2^2\right),\\
\frac{C_3b}{\lambda^m}\left(\|\varepsilon\|_{H^1}^2+b^2\left\||y|\varepsilon\right\|_2^2-s^{-(2L+\kappa)}\right)&\leq \frac{d}{ds}S(s,\varepsilon(s))
\end{align*}
hold. Then $s\mapsto S(s,\varepsilon(s))$ is monotonically increasing on $[s_\dagger,s_\ddagger]$. Therefore, we obtain
\begin{align*}
C_1C_\dagger s_\dagger^{-(2L+\kappa)}&=C_1\left(\|\varepsilon(s_\dagger)\|_{H^1}^2+b(s_\dagger)^2\|y\varepsilon(s_\dagger)\|_2^2\right)\\
&\leq \lambda(s_\dagger)^mS(s_\dagger,\varepsilon(s_\dagger))\\
&\leq \lambda(s_\dagger)^mS(s_\ddagger,\varepsilon(s_\ddagger))\\
&\leq C_2\frac{\lambda(s_\dagger)^m}{\lambda(s_\ddagger)^m}\left(\|\varepsilon(s_\ddagger)\|_{H^1}^2+b(s_\ddagger)^2\left\||y|\varepsilon(s_\ddagger)\right\|_2^2\right)\\
&\leq C_2\frac{\lambda(s_\dagger)^m}{\lambda(s_\ddagger)^m}s_\ddagger^{-(2L+\kappa)}\\
&\leq 2C_2\frac{s_\ddagger^{-(2L+\kappa-m)}}{s_\dagger^{-(2L+\kappa-m)}}s_\dagger^{-(2L+\kappa)}.
\end{align*}
Accordingly, we obtain
\[
C_1C_\dagger\leq 2C_2.
\]
It is a contradiction since $C_\dagger$ is sufficiently large.

We prove \eqref{relamesti}. Since
\[
\left|\frac{d}{ds}\left(s\lambda\right)\right|\leq s^{-1}(1+\epsilon)\left(s^{-M}+C s^{-(2L-1)}\right)\leq (1+\epsilon)s^{-(M+1)}
\]
and $\lambda(s_1)={s_1}^{-1}$, we have
\[
\left|s\lambda-1\right|\leq\int_s^{s_1}(1+\epsilon)\sigma^{-(M+1)}d\sigma\leq\frac{1+\epsilon}{M}s^{-M}.
\]
Therefore, \eqref{relamesti} holds since $M>1$. Next, we prove \eqref{rebesti}. Since
\[
\left|\frac{b}{\lambda}-1\right|\lesssim \int_s^{s_1}\sigma^{-(2L-1)}d\sigma\lesssim s^{-2(L-1)},
\]
we have
\[
\left|sb-s\lambda\right|\lesssim s^{-2(L-1)}.
\]
Consequently, we have
\[
\left|sb-1\right|\leq \left|sb-s\lambda\right|+\left|s\lambda-1\right|\leq s^{-2(L-1)}+\frac{1+\epsilon}{M}s^{-M}.
\]
Therefore, \eqref{rebesti} holds. Finally, since
\[
\left|w(s)\right|\leq \int_s^{s_1}|\Mod(\sigma)|d\sigma\lesssim \int_s^{s_1}\sigma^{-3}d\sigma\lesssim s^{-2},
\]
we obtain \eqref{rewesti}.
\end{proof}

From Lemma \ref{rebootstrap} and the definition of $s_*$, we obtain the following:

\begin{corollary}
\label{reesti}
If $s_0$ is sufficiently large, then $s_*=s'$.
\end{corollary}

\begin{lemma}
\label{s'=s_0}
If $s_0$ is sufficiently large, then $s'=s_0$.
\end{lemma}

\begin{proof}
We prove $s'\leq s_0$ by contradiction. Assume that for any $s_0\gg1$, there exists $s_1>s_0$ such that $s'>s_0$. In the following, we consider the initial value \eqref{initial} in response to such $s_1$ and the corresponding solution $u$ for \eqref{NLS}.

Let $t':=\inf I_{t_1}$. Then $s'=\inf J_{s_1}>s_0$ holds. Furthermore, we have
\[
\left\|\lambda(s)^\frac{N}{2}u(s,\lambda(s)y-w(s))e^{-i\gamma(s)}-Q(y)\right\|_{H^1}=\left\|\varepsilon(s)\right\|_{H^1}\leq\frac{\delta}{4}
\]
for all $s\in (s',s_1]$. Since $t_{t_1}((s',s_1])=(t',t_1]$, we have
\[
\left\|\tilde{\lambda}(t)^\frac{N}{2}u(t,\tilde{\lambda}(t)y-\tilde{w}(t))e^{-i\tilde{\gamma}(t)}-Q(y)\right\|_{H^1}\leq\frac{\delta}{4}
\]
for all $t\in (t',t_1]$. We consider three cases $t'>T_*$, $t'=T_*>-\infty$, and $t'=-\infty$.

Firstly, assume $t'>T_*$. Then $\lambda$ and $\tilde{\lambda}$ are bounded on $(s',s_1]$ and $(t',t_1]$, respectively, according to \eqref{bootstrap} and Corollary \ref{reesti}. Then, by setting $t$ sufficiently close to $t'$, we have
\[
\left\|\tilde{\lambda}(t)^\frac{N}{2}u(t',\tilde{\lambda}(t)y-\tilde{w}(t))e^{-i\tilde{\gamma}(t)}-Q(y)\right\|_{H^1}<\delta.
\]
Therefore, there exists the decomposition of $u$ in a neighbourhood of $t'$ according to Lemma \ref{decomposition}. Its existence contradicts the maximality of $I_{t_1}$.

Next, assume $t'=T_*>-\infty$. Then $\|\nabla u(t)\|_2\rightarrow \infty\ (t\searrow t')$ holds according to the blow-up alternative. Also, $\|\nabla u(s)\|_2\rightarrow \infty\ (s\searrow s')$ holds. Then since
\[
\|u(s)\|_2+\lambda(s)\|\nabla u(s)\|_2\lesssim 1,
\]
we have $\lambda(s)\rightarrow 0\ (s\searrow s')$. Therefore, we obtain
\[
\left|s\lambda(s)-1\right|\rightarrow 1,\quad s^{-M}\rightarrow {s'}^{-M}<\frac{1}{2}\quad (s\searrow s'),
\]
which contradicts \eqref{relamesti}.

Finally, assume $t'=-\infty$. Then there exists a sequence $(s_n)_{n\in\mathbb{N}}$ that converges to $s'$ such that $\lim_{n\rightarrow\infty}\lambda(s_n)=\infty$ holds. Therefore, we obtain
\[
\left|s_n\lambda(s_n)-1\right|\rightarrow \infty,\quad {s_n}^{-M}\rightarrow {s'}^{-M}<1\quad (n\rightarrow\infty),
\]
which contradicts \eqref{relamesti}.

Consequently, we obtain $s'\leq s_0$.
\end{proof}

\section{Conversion of estimates}
\label{sec:convesti}
In this section, we rewrite the estimates for $s$ in Lemma \ref{rebootstrap} into estimates for $t$.

\begin{lemma}[Interval]
\label{interval}
Let $s_0$ be sufficiently large. Then there exists $t_0<0$ such that 
\[
[t_0,t_1]\subset {s_{t_1}}^{-1}([s_0,s_1]),\quad \left|s_{t_1}(t)^{-1}-|t|\right|\lesssim |t|^{M+1}\quad (t\in [t_0,t_1])
\]
hold for all $t_1\in(t_0,0)$.
\end{lemma}

\begin{proof}
Firstly, $[t_{t_1}(s_0),t_1]={s_{t_1}}^{-1}([s_0,s_1])$ holds. For all $s\in[s_0,s_1]$, we have
\[
t_1-t_{t_1}(s)=s^{-1}-{s_1}^{-1}+\int_s^{s_1}\sigma^{-2}\left(\sigma\lambda_{t_1}(\sigma)+1\right)\left(\sigma\lambda_{t_1}(\sigma)-1\right)d\sigma
\]
since $-{s_1}^{-1}=t_1={t_{t_1}}(s_1)$. Therefore, we have
\[
\frac{1}{2}s^{-1}\leq s^{-1}\left(1-3s^{-M}\right)\leq\left|t_{t_1}(s)\right|\leq s^{-1}\left(1+3s^{-M}\right)\leq 2s^{-1}.
\]
Accordingly, we obtain $\frac{1}{2}\left|t_{t_1}(s)\right|\leq s^{-1}\leq 2\left|t_{t_1}(s)\right|$. According to ${s_{t_1}}^{-1}=t_{t_1}$, we obtain
\begin{align}
\label{tsesti2}
\frac{1}{2}|t|\leq s_{t_1}(t)^{-1}\leq 2|t|.
\end{align}
Consequently, according to \eqref{tsesti2}, we obtain
\[
\left||t|-s_{t_1}(t)^{-1}\right|\leq 3{s_{t_1}}(t)^{-(M+1)}\leq 3\cdot 2^{M+1}|t|^{M+1}.
\]
Furthermore, since
\[
t_{t_1}(s_0)=-|t_{t_1}(s_0)|\leq -\frac{1}{2}s_{t_1}(t_{t_1}(s_0))^{-1}=-\frac{1}{2}{s_0}^{-1}
\]
and $s_0$ is independent of $t_1$ according to Lemma \ref{s'=s_0}, we obtain Lemma \ref{interval}.
\end{proof}

\begin{lemma}[Conversion of estimates]
\label{uniesti}
For all $t\in[t_0,t_1]$, 
\begin{align*}
&\tilde{\lambda}_{t_1}(t)=|t|\left(1+\epsilon_{\tilde{\lambda},t_1}(t)\right),\quad \tilde{b}_{t_1}(t)=|t|\left(1+\epsilon_{\tilde{b},t_1}(t)\right),\quad\left|\tilde{w}_{t_1}(t)\right|\lesssim |t|^{2L-1},\\
&\|\tilde{\varepsilon}_{t_1}(t)\|_{H^1}\lesssim |t|^{L+\frac{\kappa'}{2}},\quad \||y|\tilde{\varepsilon}_{t_1}(t)\|_{2}\lesssim |t|^{L+\frac{\kappa'}{2}-1}
\end{align*}
holds for some functions $\epsilon_{\tilde{\lambda},t_1}$ and $\epsilon_{\tilde{b},t_1}$. Furthermore,
\[
\sup_{t_1\in[t,0)}\left|\epsilon_{\tilde{\lambda},t_1}(t)\right|\lesssim |t|^M,\quad \sup_{t_1\in[t,0)}\left|\epsilon_{\tilde{b},t_1}(t)\right|\lesssim |t|^{M'}.
\]
\end{lemma}

\begin{proof}
Firstly, we define $\epsilon_{\tilde{\lambda},t_1}(t):=\frac{\tilde{\lambda}_{t_1}(t)}{|t|}-1$. According to \eqref{tsesti2} and Lemma \ref{interval}, we have
\[
\left|\epsilon_{\tilde{\lambda},t_1}(t)\right|=\left|\left(s_{t_1}(t)\tilde{\lambda}_{t_1}(t)-1\right)\frac{1}{s_{t_1}(t)|t|}+\frac{1}{s_{t_1}(t)|t|}-1\right|\lesssim |t|^M.
\]
Similarly, we define $\epsilon_{\tilde{b},t_1}(t):=\frac{\tilde{b}_{t_1}(t)}{|t|}-1$ and obtain estimates of $\tilde{b}_{t_1}(t)$ and $\tilde{w}_{t_1}(t)$.
\end{proof}

\section{Proof of Theorem \ref{theorem:EMBS}}
\label{sec:proof}
\begin{proof}[Proof of Theorem \ref{theorem:EMBS}]
Let $(t_n)_{n\in\mathbb{N}}\subset(t_0,0)$ be a increasing sequence such that $\lim_{n\nearrow \infty}t_n=0$. For each $n\in\mathbb{N}$, let $u_n$ be the solution for \eqref{NLS} with the initial value
\begin{align*}
u_n(t_n,x):=\frac{1}{{\lambda_{1,n}}^\frac{N}{2}}Q\left(\frac{x}{\lambda_{1,n}}\right)e^{-i\frac{b_{1,n}}{4}\frac{|x|^2}{{\lambda_{1,n}}^2}}
\end{align*}
at $t_n$, where $b_{1,n}=\lambda_{1,n}={s_n}^{-1}=-t_n$.

According to Lemma \ref{decomposition}, there exists the decomposition
\[
u_n(t,x)=\frac{1}{\tilde{\lambda}_n(t)^{\frac{N}{2}}}\left(Q+\tilde{\varepsilon}_n\right)\left(t,\frac{x+\tilde{w}_n(t)}{\tilde{\lambda}_n(t)}\right)e^{-i\frac{\tilde{b}_n(t)}{4}\frac{|x+\tilde{w}_n(t)|^2}{\tilde{\lambda}_n(t)^2}+i\tilde{\gamma}_n(t)}
\]
on $[t_0,t_n]$. Then $(u_n(t_0))_{n\in\mathbb{N}}$ is bounded in $\Sigma^1$. Therefore, up to a subsequence, there exists $u_\infty(t_0)\in \Sigma^1$ such that
\[
u_n(t_0)\rightharpoonup u_\infty(t_0)\quad \mbox{weakly in}\ \Sigma^1.
\]
Moreover, as in Section 3.2 in \cite{LMR}, we see that
\[
u_n(t_0)\rightarrow u_\infty(t_0)\quad \mbox{in}\ L^2(\mathbb{R}^N)\quad (n\rightarrow\infty).
\]

Let $u_\infty$ be  the solution for \eqref{NLS} with the initial value $u_\infty(t_0)$ and $T^*$ be the supremum of the maximal existence interval of $u_\infty$. Moreover, we define $T:=\min\{0,T^*\}$. For any $T'\in[t_0,T)$, we have $[t_0,T']\subset[t_0,t_n]$ if $n$ is sufficiently large. Then there exists $n_0$ such that 
\[
\sup_{n\geq n_0}\|u_n\|_{L^\infty([t_0,T'],\Sigma^1)}\lesssim \left(1+|T'|^{-1}\right)\left(1+|t_0|^L\right)
\]
holds. According to Lemma \ref{contidepend}, 
\[
u_n\rightarrow u_\infty\quad \mathrm{in}\ C([t_0,T'],L^2(\mathbb{R}^N))\quad (n\rightarrow\infty)
\]
holds. In particular, $u_n(t)\rightharpoonup u_\infty(t)\ \mathrm{in}\ \Sigma^1$ for any $t\in [t_0,T)$. Furthermore, we have
\[
\|u_\infty(t)\|_2=\|u_\infty(t_0)\|_2=\lim_{n\rightarrow\infty}\|u_n(t_0)\|_2=\lim_{n\rightarrow\infty}\|u_n(t_n)\|_2=\|Q\|_2.
\]

According to weak convergence in $H^1(\mathbb{R}^N)$ and Lemma \ref{decomposition}, we decompose $u_\infty$ to
\[
u_\infty(t,x)=\frac{1}{\tilde{\lambda}_\infty(t)^{\frac{N}{2}}}\left(Q+\tilde{\varepsilon}_\infty\right)\left(t,\frac{x+\tilde{w}_\infty(t)}{\tilde{\lambda}_\infty(t)}\right)e^{-i\frac{\tilde{b}_\infty(t)}{4}\frac{|x+\tilde{w}_\infty(t)|^2}{\tilde{\lambda}_\infty(t)^2}+i\tilde{\gamma}_\infty(t)}
\]
on $[t_0,T)$. Furthermore, as $n\rightarrow\infty$, 
\begin{align*}
&\tilde{\lambda}_n(t)\rightarrow\tilde{\lambda}_\infty(t),\quad \tilde{b}_n(t)\rightarrow \tilde{b}_\infty(t),\quad \tilde{w}_n(t)\rightarrow\tilde{w}_\infty(t),\quad e^{i\tilde{\gamma}_n(t)}\rightarrow e^{i\tilde{\gamma}_\infty(t)},\\
&\tilde{\varepsilon}_n(t)\rightharpoonup \tilde{\varepsilon}_\infty(t)\quad \mbox{weakly in}\ \Sigma^1
\end{align*}
hold for any $t\in[t_0,T)$. Therefore, we obtain
\begin{align*}
&\tilde{\lambda}_{\infty}(t)=\left|t\right|(1+\epsilon_{\tilde{\lambda},0}(t)),\quad \tilde{b}_{\infty}(t)=\left|t\right|(1+\epsilon_{\tilde{b},0}(t)),\quad \left|\tilde{w}_\infty(t)\right|\lesssim |t|^{2L-1},\\
&\|\tilde{\varepsilon}_{\infty}(t)\|_{H^1}\lesssim \left|t\right|^{L+\frac{\kappa'}{2}},\quad \||y|\tilde{\varepsilon}_{\infty}(t)\|_2\lesssim \left|t\right|^{L+\frac{\kappa'}{2}-1},\quad \left|\epsilon_{\tilde{\lambda},0}(t)\right|\lesssim |t|^M,\quad \left|\epsilon_{\tilde{b},0}(t)\right|\lesssim |t|^{M'}
\end{align*}
from the uniform estimates in Lemma \ref{uniesti}. Consequently, we obtain Theorem \ref{theorem:EMBS}.
\end{proof}

\appendix
\section{A fact regarding the Schr\"{o}dinger equation}
\label{sec:SFacts}
In this section, we describe a certain continuous dependence on the initial values used in the proof of Theorem \ref{theorem:EMBS}. For notation, see \cite{CSSE}.

We consider a more general Schr\"{o}dinger equation
\begin{align}
\label{GNLS}
i\frac{\partial u}{\partial t}+\Delta u+g(u)=0,\quad (t,x)\in\mathbb{R}\times\mathbb{R}^N.
\end{align} 
For $g=g_1+\cdots+g_k$, we consider the following assumptions:
\begin{enumerate}[(a)]
\item \label{Con1}There exists $G_j\in C^1(H^1(\mathbb{R}^N),\mathbb{R})$ such that $G_j'=g_j$.
\item \label{Con2}There exist $r_j,\rho_j\in[2,2^*)$ such that for any $M<\infty$, there exists $L(M)<\infty$ such that
\[
\left\|g_j(u)-g_j(v)\right\|_{\rho_j'}\leq L(M)\|u-v\|_r
\]
for all $u,v\in H^1(\mathbb{R}^N)$ such that $\|u\|_{H^1}+\|v\|_{H^1}\leq M$.
\item \label{Con3}For any $u\in H^1(\mathbb{R}^N)$,
\[
\im g_j(u)\overline{u}=0\quad\mbox{a.e. in }\mathbb{R}^N.
\]
\end{enumerate}
Here, $p'$ is the H\"{o}lder conjugate and $2^*$ is the Sobolev conjugate, i.e., $2^*:=\frac{2N}{N-2}\ (N\geq 3)$, $2^*:=\infty\ (N=1,2)$.

\begin{lemma}
\label{contidepend}
Let $g=g_1+\cdots+g_k$ satisfy \eqref{Con1}, \eqref{Con2}, and \eqref{Con3}. For $\varphi_n$ and $\varphi\in H^1(\mathbb{R}^N)$, let $u_n$ and $u$ be solutions for \eqref{GNLS} with $u_n(0)=\varphi_n$ and $u(0)=\varphi$, respectively. Moreover, we assume that $\varphi_n\rightarrow\varphi$ in $L^2(\mathbb{R}^N)$ and that for any bounded closed interval $J\subset(-T_{\min}(\varphi),T_{\max}(\varphi))$, there exists $m\in\mathbb{N}$ such that $\sup_{n\geq m}\|u_n\|_{L^\infty(J,H^1)}<\infty$. Then
\[
u_n\rightarrow u\quad \mathrm{in}\ C(J,L^2(\mathbb{R}^N))\quad(n\rightarrow\infty).
\]
In particular, $u_n(t)\rightharpoonup u(t)$ weakly in $H^1(\mathbb{R}^N)$ for any $t\in I$.
\end{lemma}

\begin{proof}
We may assume that $T_1,T_2>0$ and $J=[-T_1,T_2]$. Then we define
\[
M:=\|u\|_{L^\infty(J,H^1)}+\sup_{n\geq m}\|u_n\|_{L^\infty(J,H^1)}.
\]
Furthermore, we define
\[
\mathcal{G}_j(u)(t):=i\int_0^t\mathcal{T}(t-s)g_j(u(s))ds,\quad \mathcal{H}(u)(t):=\mathcal{T}(t)\varphi+\mathcal{G}_1(u)(t)+\cdots+\mathcal{G}_k(u)(t).
\]
Similarly, we define $\mathcal{G}_j(u_n)$ and $\mathcal{H}(u_n)$. According to Duhamel's principle, we have $u=\mathcal{H}(u)$ and $u_n=\mathcal{H}(u_n)$.

Let $n\geq m$ and $0<T\leq\min\{T_1,T_2\}$. Moreover, let $(q,r)$, $(q_j,r_j)$, and $(\gamma_j,\rho_j)$ be admissible pairs. Then, according to the Strichartz estimate and \eqref{Con2}, we have
\begin{align*}
\|\mathcal{T}(t)\varphi_n-\mathcal{T}(t)\varphi\|_{L^q(\mathbb{R},L^r)}&\leq C\|\varphi_n-\varphi\|_{L^2},\\
\|\mathcal{G}_j(u_n)-\mathcal{G}_j(u)\|_{L^q((-T,T),L^r)}&\leq C(M)T^{\frac{1}{\gamma'_j}-\frac{1}{q_j}}\|u_n-u\|_{L^{q_j}((-T,T),L^{r_j})}.
\end{align*}

For $v,w\in C([-T,T],H^1(\mathbb{R}^N))$, we define
\[
d(v,w):=\|v-w\|_{L^\infty((-T,T),L^2)}+\sum_{j=1}^k\|v-w\|_{L^{q_j}((-T,T),L^{r_j})}.
\]
Then we have
\[
d(u_n,u)=d(\mathcal{H}(u_n),\mathcal{H}(u))\leq C\|\varphi_n-\varphi\|_{L^2}+d(u_n,u)C(M)\sum_{j=1}^kT^{\frac{1}{\gamma'_j}-\frac{1}{q_j}}.
\]
Since there exists $T(M)>0$ such that $C(M)\sum_{j=1}^kT(M)^{\frac{1}{\gamma'_j}-\frac{1}{q_j}}\leq\frac{1}{2}$, we obtain
\[
\|u_n-u\|_{L^\infty((-T(M),T(M)),L^2)}\leq d(u_n,u)\leq C\|\varphi_n-\varphi\|_{L^2}\rightarrow 0\quad (n\rightarrow \infty),
\]
which yields the conclusion.

Finally, $(u_n(t))_{n\in\mathbb{N}}$ is bounded in $H^1(\mathbb{R}^N)$ and converges to $u(t)$ in $L^2(\mathbb{R}^N)$ for any $t\in I$. Therefore, $(u_n(t))_{n\in\mathbb{N}}$ weakly converges to $u(t)$ in $H^1(\mathbb{R}^N)$.
\end{proof}

\section*{Acknowledgement}
The author would like to thank Masahito Ohta and Noriyoshi Fukaya for their support in writing this paper.

\end{document}